\documentclass[10pt,reqno]{amsart}
\usepackage{latexsym,amsmath}
\usepackage{enumerate}
\usepackage{amsfonts}
\usepackage{amssymb}
\usepackage{latexsym}

\usepackage[T1]{fontenc}
\usepackage{fourier}
\usepackage{bbm}
\usepackage{verbatim}
\usepackage{graphicx}
\usepackage{float}
\usepackage{color}
\usepackage{fullpage}

\usepackage{hyperref}
\allowdisplaybreaks[1]

\renewcommand\subset{\subseteq}

\newcommand\vm{\vec m}

\newcommand{\BP}{\mathrm{BP}}

\newcommand{\GG}{\mathbb G}

\newcommand\vU{\vec U}

\newcommand\Cutm{\cD_{\Box}}

\newcommand\nix{\,\cdot\,}

\newcommand\vI{\vec I}

\newcommand\G{\vec G}

\numberwithin{equation}{section}

\renewcommand{\vec}[1]{\boldsymbol{#1}}

\newcommand\SIGMA{\vec\sigma}

\newcommand\TAU{\vec\tau}
\newcommand\THETA{\vec\theta}

\newtheorem{definition}{Definition}[section]

\newtheorem{example}[definition]{Example}

\newtheorem{theorem}[definition]{Theorem}
\newtheorem{lemma}[definition]{Lemma}
\newtheorem{proposition}[definition]{Proposition}
\newtheorem{corollary}[definition]{Corollary}

\newtheorem{fact}[definition]{Fact}

\newcommand\cC{\mathcal{C}}
\newcommand\cD{\mathcal{D}}
\newcommand\cF{\mathcal{F}}

\newcommand\cE{\mathcal{E}}

\newcommand\cM{\mathcal{M}}
\newcommand\cO{\mathcal{O}}
\newcommand\cP{\mathcal{P}}
\newcommand\cX{\mathcal{X}}

\def\cC{{\mathcal C}}
\def\cE{{\mathcal E}}

\newcommand\w{{\omega}}
\newcommand{\beq}{\begin{equation}} \newcommand{\eeq}{\end{equation}}
\newcommand{\dc}{d_{q,\mathrm{cond}}}

\newcommand\eul{\mathrm{e}}
\newcommand\eps{\varepsilon}

\newcommand\Erw{\mathrm{E}}
\newcommand{\vecone}{\vec{1}}

\newcommand{\Po}{{\rm Po}}

\newcommand\TV[1]{\left\|{#1}\right\|_{\mathrm{TV}}}

\newcommand\bc[1]{\left({#1}\right)}
\newcommand\cbc[1]{\left\{{#1}\right\}}

\newcommand{\bck}[1]{\left\langle{#1}\right\rangle}
\newcommand\brk[1]{\left\lbrack{#1}\right\rbrack}
\newcommand\scal[2]{\bck{{#1},{#2}}}

\newcommand\abs[1]{\left|{#1}\right|}

\newcommand\RR{\mathbb{R}}

\newcommand{\whp}{a.a.s.}

\newcommand{\tensor}{\otimes}

\newcommand{\Erdos}{Erd\H{o}s}
\newcommand{\Renyi}{R\'enyi}

\newcommand{\Szemeredi}{Szemer\'edi}

\newcommand\pr{\mathrm{P}}

\newcommand\Lem{Lemma}
\newcommand\Prop{Proposition}
\newcommand\Thm{Theorem}

\newcommand\Cor{Corollary}
\newcommand\Sec{Section}
\newcommand\Chap{Chapter}

\begin{document}

\title{Bethe states of random factor graphs}

\author[Coja-Oghlan and Perkins]{Amin Coja-Oghlan, Will Perkins}

\address{Amin Coja-Oghlan, {\tt acoghlan@math.uni-frankfurt.de}, Goethe University, Mathematics Institute, 10 Robert Mayer St, Frankfurt 60325, Germany.}

\address{Will Perkins, {\tt math@willperkins.org}, Department of Mathematics, Statistics, and Computer Science, University of Illinois at Chicago, Chicago, IL, USA.}

\begin{abstract}
\noindent
We verify a key component of the replica symmetry breaking hypothesis put forward in the physics literature [M\'ezard and Montanari 2009] on random factor graph models.  For a broad class of these models we verify that the Gibbs measure can be decomposed into a moderate number of \textit{Bethe states}, subsets of the state space in which both short and long range correlations of the measure take a simple form.  Moreover, we show that the marginals of these Bethe states can be obtained from fixed points of the Belief Propagation operator. 
We derive these results from a new result on the approximation of general probability measures on discrete cubes by convex combinations of product measures.

\end{abstract}

\maketitle

\section{Introduction}\label{Sec_intro}

\noindent
Since the early 2000's there has been a great deal of interest in spin systems whose geometry is induced by a sparse random graph or hypergraph.
Physicists have been investigating these `diluted mean-field models' as models of disordered systems such as glasses~\cite{CDGS,MP1,mezard1990spin}.
Indeed, in the physics literature a comprehensive but non-rigorous approach to the study of such models called the `cavity method' has been developed~\cite{MM,MP2}.
The cavity method has since been used to put forward countless intriguing predictions on similar models that had been  studied, in some cases for decades, in other disciplines such as combinatorics, computer science or coding theory~\cite{pnas,MPZ}.
These predictions and the mathematical work that ensued have revolutionized the study of random graph models.
In several cases important specific predictions have been verified rigorously.
To name just a few examples, there is the proof of the free energy formula for the ferromagnetic Potts model on the random graph~\cite{demboPotts},
the maximum independent set size on random regular graphs~\cite{DingIS}, the $k$-SAT threshold~\cite{KostaSAT,DSS3}, the mutual information in low-density parity check codes~\cite{GMU} and the information-theoretic and algorithmic thresholds of the stochastic block model and statistical inference problems~(e.g., \cite{abbe2015detection,Banks,CKPZ,massoulie2014community,mossel2013proof,Mossel}).
In other cases the physics perspective has sparked new uses of random graph models, and this has led to breakthrough results in, e.g., coding theory and computational complexity theory (e.g.,~\cite{barbier,GSV,MWW,RichardsonUrbanke,SS}).
But while most current rigorous proofs depend on complicated,  problem-specific arguments, the physics calculations are generic.
Hence, an obvious question is whether, instead of proceeding bottom-up by verifying specific predictions one at a time, we can identify the abstract mathematical principles behind the remarkable success of the cavity method.

Its cornerstone  is the {\em replica symmetry breaking} hypothesis, according to which the Gibbs distribution of a spin system on a random graph can be decomposed into {\em Bethe states}~\cite{DemboBrazil,MM}.
These are parts of the phase space where short-range correlations can be characterized in a particularly simple way (`independent boundary times internal distribution') and long-range correlations are negligible.
Furthermore, according to the cavity method the conditional Gibbs distribution on a Bethe state is intimately tied to the graph structure.
Specifically, the conditional Gibbs marginals are determined by a fixed point of the Belief Propagation operator, an easy-to-describe operator that can be read off from the graph structure.
The main result of this paper establishes the Bethe state decomposition rigorously for a broad class of random factor graph models.
A key tool behind the proof is a further, even more general result that shows how {\em any} measure on a discrete cube can be approximated in a certain metric by a mixture of a bounded number of product measures.

To gain an intuition of the concept of Bethe states, let us summarize the physics predictions on a specific model, the Potts antiferromagnet on the 
	\Erdos-\Renyi\ graph~\cite{pnas}.
Thus, consider the random graph $\GG=\GG(n,p)$ on $n$ vertices obtained by connecting any two vertices with probability $p$ independently.
Specifically, set $p=d/n$ for a fixed $d>0$ and assume that $n$ is sufficiently large.
Let us write $u\sim v$ if the edge $\{u,v\}$ is present in $\GG$.
We remember that for a number $q\geq2$ of colors and for a given $\beta>0$ the Potts antiferromagnet on $\GG$ is the distribution defined by
	\begin{align*}
	\mu_{\GG,\beta}(\sigma)&\propto\exp\brk{-\beta\sum_{1\leq u<v\leq n,u\sim v}\vecone\{\sigma_u=\sigma_v\}}&\mbox{for }\sigma\in\{1,\ldots,q\}^n.
	\end{align*}
Thus, the model has two natural parameters: the inverse temperature $\beta$ and the average degree $d$ of the \Erdos-\Renyi\ graph.
To investigate the phase diagram we will fix $\beta$ and vary $d$.

How does $\mu_{\GG,\beta}$ evolve as $d$ increases?
Initially, for $d<1$ below the percolation threshold the random graph $\GG$ typically decomposes into connected components of size $O(\ln n)$.
Clearly, the colors assigned to vertices in different components are independent.
In effect, long-range correlations are absent.
Formally, let us write $\mu_{\GG,\beta,u,v}$ for the joint distribution of the colors of vertices $u$ and $v$ and
$\mu_{\GG,\beta,u}$ for the marginal distribution of the color of vertex $u$.
Due to the symmetry among colors $\mu_{\GG,\beta,u}$ is just the uniform distribution on $[q] = \{1,\ldots,q\}$ for all $u$; but this is a specific feature of the Potts model.
Because two randomly chosen vertices likely belong to different components of $\GG$ for $d<1$ we see that
	\begin{align}\label{eqRS}
	\lim_{n\to\infty}\frac1{n^2}\sum_{1\leq u<v\leq n}\Erw\TV{\mu_{\GG,\beta,u,v}-\mu_{\GG,\beta,u}\tensor\mu_{\GG,\beta,v}}&=0,
	\end{align}
where, of course, the expectation is over the choice of $\GG$.
 Hence, correlations decay for most vertex pairs $u,v$ on a typical random graph $\GG$ as $n\to\infty$.
In the language of the cavity method,  $\mu_{\GG,\beta}$ is {\em replica symmetric} for $d<1$.

Perhaps surprisingly, the replica symmetry condition (\ref{eqRS}) continues to hold for $d$ far beyond the percolation threshold.
But there is a critical value $\dc(q,\beta)>1$ called the {\em condensation threshold} where (\ref{eqRS}) ceases to hold.
The precise value of $\dc(q,\beta)$ was recently determined rigorously~\cite{CKPZ}.
Yet even though long-range correlations prevail beyond $\dc(q,\beta)$, 
the cavity method predicts that the measure $\mu_{\GG,\beta}$ decomposes into a moderate number of Bethe states within which long-range correlations are absent.  Each of these Bethe states is nothing more than the conditional distribution $\mu_{\GG,\beta}[\nix|S]$ on a subset $S\subset\{1,\ldots,q\}^n$ of the phase space.
Furthermore, within a Bethe state the short-range correlations take a particularly simple form: if we consider a bounded-depth `cavity' around a vertex $u$ of $\GG$, then the joint distribution of the colors inside this cavity is the internal Potts distribution on the cavity itself equipped with a random boundary condition induced by independent `messages'.
The internal distribution of the cavity is extremely simple because on a sparse random graph the bounded-depth neighborhood of a vertex is typically acyclic.
Additionally, the messages that induce the boundary condition are predicted to satisfy a set of recurrence relations that makes them amenable to a rigorous analysis, the Belief Propagation equations.

To introduce Bethe states precisely we need a bit of notation. 
For a vertex $u$ and an integer $r\geq0$, let $\nabla_{u,r}$ be the depth-$r$ neighborhood of $u$; that is, $u$ and all vertices within graph distance $r$ of $u$. Let $\GG[\nabla_{u,r}]$ be the subgraph induced on $\nabla_{u,r}$.
Moreover, for an edge $\{v,w\}$ of $\GG$, a color $\omega \in\{1,\ldots,q\}$ and a set $S \subseteq\{1,\ldots,q\}^n$, define the \textit{message} from $w$ to $v$ as 
	\begin{equation}\label{eqmsg}
	\mu_{\GG,\beta,w \to v} [\omega|S] = \mu_{\GG -\{ v,w\},w, \beta}[\omega |S];
	\end{equation}
that is, the probability that $w$ receives color $\omega$ in the Potts model on the graph $\GG$ with the edge $\{v,w\}$ removed, conditional on the colorings in the set $S$.
Further, we can define a local measure on colorings of $\nabla_{v,r}$ where we combine the local Potts model on the subgraph $\GG[\nabla_{v,r}]$ with an independent boundary condition given by the messages (\ref{eqmsg}): for every $\sigma \in\{1,\ldots,q\}^{\nabla_{u,r}}$ we define
	\begin{align}\label{eqPottsBethe}
	\bar\mu_{\GG,\beta,u,r}[\sigma|S] &\propto 
	\mu_{\GG[\nabla_{u,r}],\beta}(\sigma)
	 \prod_{\substack{v \in \nabla_{u,r}\\w \notin\nabla_{u,r}\\ v\sim w}}\bc{1-(1-\eul^{-\beta})\mu_{\GG,\beta,w\to v}[\sigma_v|S]}.
	\end{align}
Additionally, let $\mu_{\GG,\beta, \nabla_{u,r}}$ be the joint distribution of the spins in $\nabla_{u,r}$ under $\mu_{\GG,\beta}$.
Then $S$ is an  $(\eps, r)$-\textit{Bethe state} if
	\begin{align}
	\label{eqbethepotts1}
	\frac{1}{n} \sum_{u=1}^n \TV{\bar \mu_{\GG,\beta,u,r}[ \nix |S] - \mu_{\GG,\beta, \nabla_{u,r}}[ \nix |S]     }  &< \eps \,\qquad\mbox{and}\\
	\label{eqRSpotts}
	\frac1{n^2}\sum_{1\leq v<w\leq n}\TV{\mu_{\GG,\beta,v,w}[\nix|S]
		-\mu_{\GG,\beta,v}[\nix|S]\tensor \mu_{\GG,\beta,w}[\nix|S]} &<\eps \, .
	\end{align}
In words, \eqref{eqbethepotts1} provides that for almost all $u$ the actual  joint distribution $\mu_{\GG,\beta, \nabla_{u,r}}$ of the colors inside the `cavity' $\nabla_{u,r}$ is close to
the idealized distribution (\ref{eqPottsBethe}) with an independent boundary condition.
Moreover, \eqref{eqRSpotts} is the exact analogue of~\eqref{eqRS} for the conditional measure $\mu_{\GG,\beta}[\nix |S]$. The condition guarantees the absence of long-range correlations within the Bethe state.

Finally, writing $\partial v$ for the set of neighbors of a vertex $v$, we say that the messages defined in (\ref{eqmsg}) form an $\eps$-\textit{Belief Propagation fixed point} if 
\begin{align}\label{eqPottsBP}
\frac{1}{n} \sum_{v\sim w}\sum_{\omega=1}^q \TV{ \mu_{\GG,\beta,v \to w}[\omega|S]-	
	\frac{\prod_{u\in\partial v\setminus w}1-(1-\eul^{-\beta})\mu_{\GG,\beta,u \to v}[\omega|S]}
		{\sum_{\sigma=1}^q\prod_{u\in\partial v\setminus w}1-(1-\eul^{-\beta})\mu_{\GG,\beta,u \to v}[\sigma|S]}}&<\eps.
\end{align}
This means that the messages on the random graph $\GG$ given $S$ approximately satisfy the same recurrence that holds for the messages of the Potts model on a tree.
In other words, given $S$ the messages that induce the local distributions (\ref{eqPottsBethe}) are accessible to the extent that we can extract information about them from the fixed point relation (\ref{eqPottsBP}), which is defined purely combinatorially in terms of the graph structure.

The following theorem, which is the specialization of our main result (Theorem~\ref{Thm_BetheStates} below) to the Potts model,
confirms the Bethe state prediction for the Potts model on the \Erdos-\Renyi\ graph.

\begin{theorem}\label{thmPotts}
For every $q\geq2,d>0, \beta>0$ and for any $r>0$ and any sequence $L=L(n)\to\infty$ there exists $\eps=\eps(n)\to0$
such that the random factor graph $\GG$ has the following property with probability $1-o(1)$ as $n\to\infty$.
There exist pairwise disjoint non-empty sets $S_1=S_1(\GG),\ldots,S_l=S_l(\GG)\subset\{1,\ldots,q\}^{n}$ with $1\leq l\leq L$ such that
	\begin{enumerate}[(i)]
	\item $\sum_{i=1}^l\mu_{\GG}(S_i)>1-\eps$,
	\item $S_i$ is an $(\eps,r)$-Bethe state of $\GG$ for each $i=1, \ldots l$, 
	\item the messages $(\mu_{\GG,\beta,v \to w}[\nix|S_i])_{v\sim w}$ are an $\eps$-Belief Propagation fixed point for every $i=1,\ldots,l$.
	\end{enumerate}
\end{theorem}

\noindent
Crucially, \Thm~\ref{thmPotts} is universal: the promised decomposition exists for {\em all} values of the model parameters $q,d,\beta$.
This is remarkable because according to physics predictions~\cite{pnas} the nature of the decomposition changes substantially as we increase $d$ (for a fixed $\beta>0$).
Specifically, the results from~\cite{CKPZ} imply that for $d$ up to condensation threshold $\dc(q,\beta)$ there is just a single Bethe state $S_1=\{1,\ldots,q\}^n$.
However, for $d>\dc(q,\beta)$ we expect that an {\em unbounded} number of Bethe states will be required~\cite{Marinari}.
Hence the need for the growing function $L(n)$ in \Thm~\ref{thmPotts}.
In fact, for $d>\dc(q,\beta)$ two different scenarios called `one-step replica symmetry breaking' and `full replica symmetry breaking' in physics jargon are deemed to be possible.
Yet despite these rather intricate (and parameter and model-dependent) predictions, \Thm~\ref{thmPotts} verifies that there is a universal structure present in all these 
regimes, and that it is merely the specific instantiation of this universal structure that varies.
In fact, the Bethe state decomposition is unique in a certain sense that can be quantified in terms of the cut metric, which we will introduce in the following section.

While in the Potts model the marginals of the unconditional measure $\mu_{\GG,\beta}$ are uniform, the marginals $\mu_{\GG,\beta,i}[\nix|S_l]$ within the Bethe states are generally not.
But Theorem~\ref{thmPotts} in particular confirms that the marginals within the Bethe states are given by Belief Propagation fixed points (to see this consider cavities $\nabla_{u,0}$ consisting just of the root vertex $u$).
This goes some way towards explaining the success of the quantitative physics computations, which are based on pinpointing Belief Propagation fixed points.
That said, depending on the specific parameter values there will generally be Belief Propagation fixed points that do not correspond to actual Bethe states and thus additional, model-specific arguments may be necessary to single out the relevant fixed points.

After presenting the abstract results on the cut metric and the decomposition of discrete measures in \Sec~\ref{Sec_subcube},
we will present the main result of the paper regarding the Bethe state decomposition for general random factor graph models in \Sec~\ref{Sec_BetheStates}.

\subsection*{Notation and preliminaries}
For an integer $l>0$ we let $[l]=\{1,\ldots,l\}$.
Moreover, for a finite set $\cX\neq\emptyset$ we denote by $\cP(\cX)$ the set of probability measures on $\cX$.
We identify $\cP(\cX)$ with the standard simplex in $\RR^{\cX}$ endowed with the total variation norm $\TV\nix$.
Further, $\cP^2(\cX)$ signifies the set of probability measures on $\cP(\cX)$ equipped with the weak topology.
The Wasserstein $\ell_1$-distance of two probability distributions $p,q\in\cP^2(\cX)$ is denoted by $\cD_1(p,q)$.

Suppose that $\Omega,V\neq\emptyset$ are finite sets.
For a distribution $\mu\in\cP(\Omega^V)$ and a set $I\subset V$ we denote by $\mu_I$ the joint distribution of the coordinates $I$.
That is,
	\begin{align*}
	\mu_I(\sigma)&=\sum_{\tau\in\Omega^n}\vecone\{\forall i\in I:\tau_i=\sigma_i\}\mu(\tau)&(I\subset V,\sigma\in\Omega^I).
	\end{align*}
If $I=\{i_1,\ldots,i_l\}$ we usually write $\mu_{i_1,\ldots,i_l}$ instead of $\mu_{\{i_1,\ldots,i_l\}}$.
Further, $\SIGMA^\mu,\TAU^\mu,\SIGMA^{1,\mu},\SIGMA^{2,\mu},\ldots\in\Omega^V$ denote independent samples drawn from $\mu$.
Where the context permits it we omit the superscript $\mu$.
Moreover, if $f:(\Omega^V)^k\to\RR$ is a function, then we use the notation
	$$\scal{f(\SIGMA^1,\ldots,\SIGMA^k)}\mu=\sum_{\sigma^1,\ldots,\sigma^k\in\Omega^n}f(\sigma^1,\ldots,\sigma^k)\prod_{j=1}^k\mu(\sigma^j)$$
for the average of $f$ over independent samples from $\mu$.

Additionally, if $I\subset V$ and $\tau\in\Omega^V$, then we let $\tau|_I=(\tau_i)_{i\in I}$ be the restriction of $\tau$ to $I$.
Moreover, for $\sigma\in\Omega^I$ we let
	$$S^{I,\sigma}=\cbc{\tau\in\Omega^V:\tau|_I=\sigma}$$
be the a sub-cube of the discrete cube $\Omega^V$ where the entries of the coordinates in $I$ coincide with $\sigma$.
Further, if $\mu\in\cP(\Omega^V)$ and $\mu(S^{I,\sigma})>0$, then
	\begin{align}
	\label{eqCondMu}
\mu^{I,\sigma}=\mu[\nix|S^{I,\sigma}]
\end{align}
is the corresponding conditional distribution of $\mu$.
If $\mu(S^{I,\sigma})=0$, then $\mu^{I,\sigma}$ is just the uniform distribution on $S^{I,\sigma}$.

Finally, guided by (\ref{eqRS}) and following~\cite{Victor}, we say that $\mu\in\cP(\Omega^V)$ is {\em $\eps$-symmetric} if 
	\begin{align}\label{eqepssym}
	\frac1{|V|^2}\sum_{i,j\in V:i\neq j}\TV{\mu_{i,j}-\mu_i\tensor\mu_j}&<\eps.
	\end{align}
In words, the spins assigned to $i,j$ are asymptotically independent for most pairs of coordinates.

\section{The cut metric}\label{Sec_subcube}

\subsection{Approximation by mixtures of product measures}
To pave the way for the study of Bethe states in the next section, in this section we take a more abstract view.
The aim is to investigate the general structure of probability measures $\mu$ on a discrete cube $\Omega^n$ for a finite set of size $|\Omega|>1$ and a large integer $n$.
Of course, such measures are ubiquitous in mathematical physics; the Gibbs measure $\mu_{\GG,\beta}$ is a concrete example.

Informally, the main result of this section shows that {\em any} probability measure $\mu\in\cP(\Omega^n)$ can be approximated well by a bounded number of product measures in a natural, canonical way.
There appear two related results in the literature, and the contribution of the main result of this section is to unify these two viewpoints.
The first of the two is the following theorem.

\begin{theorem}[{\cite{Victor}}]\label{Thm_states}
For any set $\Omega$ of size $1<|\Omega|<\infty$, and for any $\eps>0$ there exist $L>0$ and $n_0>0$ such that for all $n>n_0$ and all $\mu\in\cP(\Omega^n)$ the following is true.
There exist pairwise disjoint sets $S_1,\ldots,S_l\subset\Omega^n$, $1\leq l\leq L$, such that
	\begin{enumerate}[(i)]
	\item $\sum_{i=1}^l\mu(S_i)>1-\eps$,
	\item $\mu(S_i)>0$ and the conditional distribution $\mu[\nix|S_i]$ is  $\eps$-symmetric for each $1\leq i\leq l$.
	\end{enumerate}
\end{theorem}

\noindent
Thus, for any $\mu\in\cP(\Omega^n)$ we can chop the cube $\Omega^n$ into a {\em bounded} number of subsets $S_1,\ldots,S_l$ that cover almost the entire probability mass such that the conditional measures $\mu[\nix|S_i]$ are $\eps$-symmetric.
\Thm~\ref{Thm_states} was proved in~\cite{Victor} by an argument akin to the proof of the \Szemeredi\ regularity lemma, an important result in extremal combinatorics~\cite{Szemeredi}.
In effect, while the number $L$ of parts is independent of $n$, in terms of $\eps$ the bound on $L$ grows like a power tower of height $1/\eps^c$ for some constant $c\geq1$.
The second viewpoint is summarized by

\begin{lemma}[{\cite[\Lem~3.5]{CKPZ}}]\label{Lemma_pinning}
For any set $\Omega$ of size $1<|\Omega|<\infty$ and any $\eps>0$ there exist $n_0>0$ and a random variable $0<\THETA\leq2\eps^{-4}\ln|\Omega|$ such that for all $n>n_0$ and all $\mu\in\cP(\Omega^n)$ the following holds.
Let $\vI\subset[n]$ be a uniform subset of size $\THETA$ and choose $\SIGMA\in\Omega^{\vI}$ from $\mu_{\vI}$.
	Then  $\pr\brk{\mu^{\vI,\SIGMA}\mbox{ is $\eps$-symmetric}}>1-\eps.$
\end{lemma}

\noindent
\Lem~\ref{Lemma_pinning} shows that a certain random perturbation turns {\em any} measure $\mu\in\cP(\Omega^n)$ into an $\eps$-symmetric distribution
with high probability.
Specifically, the perturbation is to `pin'  merely a  bounded number of coordinates to the specific values observed under one random sample $\SIGMA$.
The bound on the number $\THETA$ of coordinates pinned is polynomial in $1/\eps$, tiny by comparison to a power tower.
\Thm~\ref{Thm_states} and \Lem~\ref{Lemma_pinning} have been vital ingredients in prior rigorous work on diluted mean-field models, particularly the derivation of the condensation threshold in the Potts antiferromagnet on the \Erdos-\Renyi\ graph in~\cite{CKPZ}.

But how do \Thm~\ref{Thm_states} and \Lem~\ref{Lemma_pinning} fit together?
For instance, are the measures $\mu^{\vI,\SIGMA}$ obtained from the pinning operation related to the conditional distributions $\mu[\nix|S_i]$ on the parts $S_i$ of the decomposition from \Thm~\ref{Thm_states} at all?
If so, how do the very different bounds (power tower vs.\ polynomial) reconcile?
Specifically, once we fix the set $\vI$, there is only a relative small number of $\exp(1/\eps^{O(1)})$ of possible configurations $\SIGMA$ on $\Omega^{\vI}$.
Does this mean that for a given $\vI$, we miss out on some of the parts of the partition from \Thm~\ref{Thm_states}?
Or perhaps the perturbed measures $\mu^{\vI,\SIGMA}$ amalgamate the parts of the decomposition from \Thm~\ref{Thm_states} in some way?
Besides, is the partition from \Thm~\ref{Thm_states} unique in any sense, or are there totally different ways of chopping up the set $\Omega^n$?

The main result of this section, \Thm~\ref{Thm_pin} below, shows that indeed any two decompositions of the type promised by \Thm~\ref{Thm_states} are close to each other in a certain metric, that the power tower bound for \Thm~\ref{Thm_pin} is far too pessimistic, and that the measures $\mu^{\vI,\SIGMA}$ produced by \Lem~\ref{Lemma_pinning} can indeed be identified with the parts of the decomposition promised by \Thm~\ref{Thm_pin}.
Furthermore, and this will be crucial in the proof of the Bethe state theorem in \Sec~\ref{Sec_BetheStates}, the decomposition in \Thm~\ref{Thm_states} can be chosen to be a decomposition into canonical `subcubes' of the cube $\Omega^n$.
Thus, the sets $S_1,\ldots,S_l$ have a very simple combinatorial description, and the decomposition does not actually depend much on the measure $\mu$.

To state the result formally, we introduce a metric on $\cP(\Omega^n)$, the {\em cut metric}.
Let $\Gamma(\mu,\nu)$ be the set of all couplings of probability measures $\mu,\nu\in\Omega^n$.
Thus, $\Gamma(\mu,\nu)$ contains all probability measures $\gamma$ on $\Omega^{2n}$ such that
	\begin{align*}
	\mu(\sigma)&=\sum_{\tau\in\Omega^n}\gamma(\sigma,\tau) \quad \text{ and } \quad
		\nu(\sigma)=\sum_{\tau\in\Omega^n}\gamma(\tau,\sigma)\qquad\mbox{ for every $\sigma\in\Omega^n$}.
	\end{align*}
Then we define the {\em cut metric} by
	\begin{align}\label{eqCutm}
	\Cutm(\mu,\nu)&=\frac1n\min_{\gamma\in\Gamma(\mu,\nu)}\max_{\substack{I\subset\{1,\ldots,n\}\\
		B\subset\Omega^{n}\times\Omega^n\\\omega\in\Omega}}\abs{\sum_{i\in I}\sum_{(\sigma,\tau)\in B}\gamma(\sigma,\tau)
			(\vecone\{\sigma_i=\omega\}-\vecone\{\tau_i=\omega\})}\qquad\mbox{for }\mu,\nu\in\cP(\Omega^n).
	\end{align}
Thus, we first attempt to align the measures $\mu,\nu$ as best as possible by choosing a coupling $\gamma$.
Then an adversary searches for the worst possible discrepancy under this coupling.
That is, the adversary looks for an event $B\subset\Omega^{n}\times\Omega^n$ and a set $I\subset[n]$ of coordinates where the average frequency of a specific spin $\omega$ is as different as possible under the coupling $\gamma$.
We defer the proof of the following simple fact to \Sec~\ref{Sec_metric}.

\begin{fact}\label{Fact_metric}
$\Cutm(\nix,\nix)$ is a metric on $\cP(\Omega^n)$.
\end{fact}

The cut metric is much weaker than other standard metrics on $\cP(\Omega^n)$ such as the total variation distance.
Nonetheless, we shall see momentarily that interesting random variables that, guided by~\cite{Marinari}, we call `intensive observables', behave continuously under the cut metric.
But let us first state the result about the approximation by mixtures of product measures.
In (\ref{eqCondMu}) we defined $\mu^{I,\sigma}$ to be the conditional distribution of $\mu$ on the subcube $S^{I,\sigma}$.
Additionally, let
	\begin{align*}
	\bar\mu^{I,\sigma}&=\bigotimes_{i=1}^n\mu^{I,\sigma}_i
	\end{align*}
be the product measure with the same marginals as $\mu^{I,\sigma}$.

\begin{theorem}\label{Thm_pin}
There is an absolute constant $c>0$ such that
for any set $\Omega$ of size $1<|\Omega|<\infty$ and any $0<\eps<1/2$ there exist $n_0>0$ and a random variable $0<\THETA<((2\ln |\Omega|)/\eps)^c$ such that for all $n>n_0$ and all $\mu\in\cP(\Omega^n)$ the following statements are true.
\begin{enumerate}[(i)]
\item Let $\vI\subset[n]$ be a uniformly random subset of size $\THETA$ and choose $\SIGMA\in\Omega^{\vI}$ from $\mu_{\vI}$.
	Then 
		$
		\Erw[\Cutm(\mu^{\vI,\SIGMA},\bar\mu^{\vI,\SIGMA})]<\eps.
		$
\item Further, let 
		$\bar\mu^{\vI}=\Erw[\bar\mu^{\vI,\SIGMA}\mid\vI]\in\cP(\Omega^n).$
	Then 
	$\Erw[\Cutm(\mu,\bar\mu^{\vI})]<\eps.$
\end{enumerate}
\end{theorem}

Crucially, as in \Lem~\ref{Lemma_pinning} the random variable $\THETA$ in \Thm~\ref{Thm_pin} is bounded by a number that depends on $\eps,\Omega$ but not on $n$ or $\mu$.
Thus, (i) shows that the distribution obtained from $\mu$ by merely pinning a bounded number of randomly chosen coordinates to the specific spins observed under a random configuration $\SIGMA$ is close to a product measure in the cut metric.
Additionally, (ii) shows that for a typical choice of $\vI$ the measure $\mu$ itself is close to a mixture of a bounded number of product measures 
	$$\Erw[\bar\mu^{\vI,\SIGMA}\mid\vI]=\sum_{\sigma\in\Omega^{\vI}}\mu_{\vI}(\sigma)\bar\mu^{\vI,\sigma}$$
	in the cut metric.
In fact, these product measures are obtained in a very simple, canonical way:
we just go over all $|\Omega|^{\THETA}$  possible spin configurations $\sigma$ on the coordinates $\vI$ and add in the product measure $\bar\mu^{\vI,\sigma}$ with weight $\mu_{\vI}(\sigma)$.
In particular, the number of product measures in this mixture is bounded by $\exp(1/\eps^{O(1)})$, rather than by a power tower as in \Thm~\ref{Thm_states}.

While \Thm~\ref{Thm_states} and \Lem~\ref{Lemma_pinning} are based on the concept of $\eps$-symmetry, 
\Thm~\ref{Thm_pin} speaks of cut metric approximations.
The following proposition shows that these two concepts are intimately related.

\begin{proposition}\label{Cor_symmetry}
For any $\Omega$ of size $1<\abs\Omega<\infty$, any $0<\eps<1/2$ and  any $k\geq2$ there exists $n_0>0$ such that for all $n>n_0$ and all $\mu\in\cP(\Omega^n)$ the following two statements hold. 
\begin{enumerate}[(i)]
\item If $\mu$ is $(\eps/9)^3$-symmetric, then $\Cutm(\mu,\bigotimes_{i=1}^n\mu_i)<\eps$.
\item If  $\Cutm(\mu,\bigotimes_{i=1}^n\mu_i)<\eps^4/(128|\Omega|)^{4k}$, then 
	$\sum_{1\leq i_1<\cdots<i_k\leq n}\TV{\mu_{i_1,\ldots,i_k}-\mu_{i_1}\tensor\cdots\tensor\mu_{i_k}}<\eps n^k.$
\end{enumerate}
\end{proposition}

\noindent
Combining \Thm~\ref{Thm_pin} and \Prop~\ref{Cor_symmetry}, we conclude that \Thm~\ref{Thm_pin} provides a decomposition with the same properties as \Thm~\ref{Thm_states}, but with an exponential rather than power tower bound.

Conversely, the following result shows that any decomposition like in \Thm~\ref{Thm_states} yields a cut norm approximation.

\begin{proposition}\label{Cor_states}
Suppose that $\Omega$ is a set of size $1<\abs\Omega<\infty$, that $0<\eps<1/2$, $\mu\in\cP(\Omega^n)$ for a large enough $n>n_0(\eps,\Omega)$ and that $S_1,\ldots,S_l\subset\Omega^n$ are pairwise disjoint sets such that
	\begin{enumerate}[(i)]
	\item $\sum_{i=1}^l\mu(S_i)>1-\eps$,
	\item $\mu(S_i)>0$ and the conditional distribution $\mu[\nix|S_i]$ is  $(\eps/9)^3$-symmetric for each $1\leq i\leq l$.
	\end{enumerate}
Let $z=\sum_{h=1}^l\mu(S_h)$.
Then
	\begin{align*}
	\Cutm\bc{\mu,\frac1z\sum_{i=1}^l\mu(S_i)\bigotimes_{j=1}^n\mu_j[\nix|S_i]}<2\eps.
	\end{align*}
\end{proposition}

Finally, what other conclusions can be drawn from the cut metric approximation?
Let us call a function
	$$f:(\Omega^n)^k\to[0,1]$$ a {\em $(k,l)$-intensive observable} if there exist 
$I_1,\ldots,I_l\subset\{1,\ldots,n\}$ and $\tau_1^{(j)},\ldots,\tau_l^{(j)}\in\Omega$, $j=1,\ldots,k$, such that
	\begin{align*}
	f(\sigma^{(1)},\ldots,\sigma^{(k)})&=
	\frac1{n^l}\sum_{i_1\in I_1,\ldots,i_l\in I_l}\prod_{j=1}^k\vecone\{\sigma^{(j)}_{i_1}=\tau_1^{(j)},\ldots,\sigma^{(j)}_{i_l}=\tau_l^{(j)}\}.
	\end{align*}
Then as a consequence of \Thm~\ref{Thm_pin} we obtain

\begin{corollary}\label{Cor_pin}
For any $\Omega$, $\eps>0$ and any $k,l$ there exist $\delta>0$ and $n_0>0$ such that for all $n>n_0$ and all $\mu,\nu\in\cP(\Omega^n)$ the following is true.
If $\Cutm(\mu,\nu)<\delta$, then for every $(k,l)$-intensive observable $f$ we have
	$$\abs{\scal{f(\SIGMA^{(1)},\ldots,\SIGMA^{(k)})}{\mu}-\scal{f(\SIGMA^{(1)},\ldots,\SIGMA^{(k)})}{\nu}}<\eps.$$
\end{corollary}

\noindent
In other words, intensive observables are uniformly continuous with respect to the cut metric.

A further random variable that behaves continuously with respect to the cut metric is the overlap, which plays a key role in the physicists' cavity method.
The {\em overlap of two configurations} $\sigma,\tau\in\Omega^n$ is defined as the probability distribution $\rho_{\sigma,\tau}(\omega,\omega')$ on $\Omega\times\Omega$ with
	$$\rho_{\sigma,\tau}(\omega,\omega')=\frac1n\sum_{i=1}^n\vecone\{\sigma_i=\omega,\tau_i=\omega'\}.$$
Furthermore, the {\em overlap  $\cO_\mu$ of a measure} $\mu\in\cP(\Omega^n)$ is the distribution of $\rho_{\SIGMA,\TAU}$ with $\SIGMA,\TAU$ chosen independently from $\mu$.
Thus, $\cO_\mu\in\cP^2(\Omega\times\Omega)$.
As another application of \Thm~\ref{Thm_pin} we obtain

\begin{corollary}\label{Cor_overlap}
For any $\Omega$ and any $\eps>0$ there exist $\delta>0$ and $n_0>0$ such that for all $n>n_0$ and all $\mu,\nu\in\cP(\Omega^n)$ the following is true.
If $\Cutm(\mu,\nu)<\delta$, then $\cD_1(\cO_\mu,\cO_\nu)<\eps$.
\end{corollary}

After proving Fact~\ref{Fact_metric} in \Sec~\ref{Sec_metric} we establish \Thm~\ref{Thm_pin} in \Sec~\ref{sec:PinningProof}.
The proof, which is technically not all too difficult, builds upon \Lem~\ref{Lemma_pinning} and a few ideas from  prior work on the decomposition of discrete measures~\cite{Victor,Limits}.
The remaining subsections contain the proofs of 
the remaining results.

\medskip\paragraph{\em Related work}
The cut metric was originally introduced as a metric on matrices and graphs~\cite{FK} and it was used to derive a `weak' version of \Szemeredi's regularity lemma for graphs~\cite{Szemeredi} that, though still useful to approximate various graph parameters of interest, gets by with partitioning the vertex set into a smaller number of classes than required by the original version.
Variants of the cut metric have since come to play an eminent role in the theory of graph limits~\cite{Borgs, Janson,LovSz}.
At least superficially the cut metric may seem reminiscent of the transport metric because both involve a minimization over couplings; but
the cut metric is more involved due to the inner maximization problem which depends on the coupling. 
We are not aware of a prior reference for the cut metric for probability distributions as defined in (\ref{eqCutm}), but the definition is
closely related to the metric used in recent work on a limiting theory for discrete probability distributions~\cite{Limits}.
There the spaces $\cP(\Omega^n)$, $n\geq1$, are all embedded into a `continuous' Polish space that is equipped with a continuous version of the cut metric.
The discrete version (\ref{eqCutm}) and the continuous version are equivalent to the extent that if we were to consider sequences $(\mu_n)_n$ of probability distributions $\mu_n\in\cP(\Omega^n)$, then the cut metric from (\ref{eqCutm}) and the metric from~\cite{Limits} would yield the same Cauchy sequences.
Some of the statements in this section (e.g., \Prop~\ref{Cor_symmetry} and \Lem~\ref{Lemma_prod}) have analogues in the limiting theory from~\cite{Limits}.
The added value of the discrete versions is that we obtain uniformity in $n$ and/or explicit rates such as the expressions given in \Prop~\ref{Cor_symmetry}.
Finally, the cut metric and the limiting theory from~\cite{Limits} is related to the Aldous-Hoover representation, a connection that was already pointed out by Panchenko~\cite{PanchenkoBook}.

The proof of \Lem~\ref{Lemma_pinning} in~\cite{CKPZ} hinges on an information-theoretic inequality that goes back to the work of Montanari~\cite{Andrea}.
Results very similar to \Lem~\ref{Lemma_pinning} were discovered independently in the study of convex optimization problems community~\cite{Raghavendra}.
Although we have not tried to optimize the constant $c$ in \Lem~\ref{Lemma_pinning}, a natural question is what the best dependence on $\eps$ might be, cf.~\cite{Allen}.

\subsection{Proof of Fact~\ref{Fact_metric}}\label{Sec_metric}
Merely the triangle inequality needs verifying.
We use an argument similar to the one from~\cite{Janson}.
Suppose that $\mu,\nu,\eta\in\cP(\Omega^n)$ and let $\gamma\in\Gamma(\mu,\nu),\gamma'\in\Gamma(\nu,\eta)$ be couplings for which the cut distance is attained.
We define a coupling $g\in\Gamma(\mu,\eta)$ by letting
	\begin{align*}
	g(\sigma,\tau)&=\sum_{\chi\in\Omega^n:\nu(\chi)>0}\frac{\gamma(\sigma,\chi)\gamma'(\chi,\tau)}{\nu(\chi)}.
	\end{align*}
Now let $B\subset\Omega^n\times\Omega^n$ and $I\subset[n]$.
Then for any $\omega\in\Omega$,
	\begin{align}\nonumber
	\sum_{i\in I}\sum_{(\sigma,\tau)\in B}g(\sigma,\tau)&(\vecone\{\sigma_i=\omega\}-\vecone\{\tau_i=\omega\})\\&=
	\sum_{(\sigma,\tau)\in B}\sum_{\chi\in\Omega^n:\nu(\chi)>0}\frac{\gamma(\sigma,\chi)\gamma'(\chi,\tau)}{\nu(\chi)}
		\sum_{i\in I}(\vecone\{\sigma_i=\omega\}-\vecone\{\chi_i=\omega\}+\vecone\{\chi_i=\omega\}-\vecone\{\tau_i=\omega\})\nonumber\\
	&=\sum_{\sigma\in\Omega^n}\sum_{\chi\in\Omega^n:\nu(\chi)>0}\sum_{\tau\in\Omega^n:(\sigma,\tau)\in B}\frac{\gamma(\sigma,\chi)\gamma'(\chi,\tau)}{\nu(\chi)}\sum_{i\in I}(\vecone\{\sigma_i=\omega\}-\vecone\{\chi_i=\omega\})\nonumber\\
		&\qquad+\sum_{\tau\in\Omega^n}\sum_{\chi\in\Omega^n:\nu(\chi)>0}\sum_{\sigma\in\Omega^n:(\sigma,\tau)\in B}\frac{\gamma(\sigma,\chi)\gamma'(\chi,\tau)}{\nu(\chi)}\sum_{i\in I}(\vecone\{\chi_i=\omega\}-\vecone\{\tau_i=\omega\}).
			\label{eqFact_metric1}
	\end{align}
Assume without loss that the sum on the l.h.s.\ is positive (otherwise we swap $\mu,\nu$) and let
	\begin{align*}
	B_1&=\textstyle\cbc{(\sigma,\chi)\in\Omega^n\times\Omega^n:
		\sum_{i\in I}\vecone\{\sigma_i=\omega\}-\vecone\{\chi_i=\omega\}>0},&
	B_2&=\textstyle\cbc{(\chi,\tau)\in\Omega^n\times\Omega^n:
		\sum_{i\in I}\vecone\{\chi_i=\omega\}-\vecone\{\tau_i=\omega\}>0}.
	\end{align*}
Then (\ref{eqFact_metric1}) yields
	\begin{align*}
	\sum_{i\in I}\sum_{(\sigma,\tau)\in B}g(\sigma,\tau)&(\vecone\{\sigma_i=\omega\}-\vecone\{\tau_i=\omega\})\\
		&\leq\sum_{i\in I}\sum_{(\sigma,\chi)\in B_1}\gamma(\sigma,\chi)(\vecone\{\sigma_i=\omega\}-\vecone\{\tau_i=\omega\})
			+\sum_{i\in I}\sum_{(\chi,\tau)\in B_2}\gamma(\chi,\tau)(\vecone\{\chi_i=\omega\}-\vecone\{\tau_i=\omega\})\\
			&\leq n(\Cutm(\mu,\nu)+\Cutm(\nu,\eta)).
	\end{align*}
Since this bound holds for all $B,I,\omega$, the triangle inequality follows.

\subsection{Proof of \Thm~\ref{Thm_pin}}\label{sec:PinningProof}

\noindent
We will derive \Thm~\ref{Thm_pin} from \Lem~\ref{Lemma_pinning}.
Let us begin with the following observation.

\begin{lemma}\label{Lemma_cutm}
For any $\eps>0$, $\Omega$ there is 
$n_0>0$ such that for all $n>n_0$ the following is true.
Assume that $\nu\in\cP(\Omega^n)$ is $(\eps/9)^3$-symmetric.
Then $\Cutm(\nu,\bigotimes_{i=1}^n\nu_i)<\eps.$
\end{lemma}
\begin{proof}
Let $\delta=(\eps/9)^3$.
To upper-bound the cut distance we couple $\nu$ and $\bar\nu=\bigotimes_{i=1}^n\nu_i$ independently;
that is, the coupling $\gamma\in\cP(\Omega^n\times\Omega^n)$ is just the product distribution $\gamma=\nu\tensor\bar\nu$.
Thus,  we need to show that
	\begin{align}\label{eqLemma_cutm1}
	\Cutm(\nu,\bar\nu)&\leq\frac1n\max_{\substack{I\subset\{1,\ldots,n\}\\
		B\subset\Omega^{2n}\\\omega\in\Omega}}\abs{\sum_{i\in I}\sum_{(\sigma,\tau)\in B}\nu(\sigma)\bar\nu(\tau)
			(\vecone\{\sigma_i=\omega\}-\vecone\{\tau_i=\omega\})}<\eps.
	\end{align}
For a set $I\subset[n]$ and $\omega\in\Omega$ consider the random variable $X_{I,\omega}(\SIGMA)=\sum_{i\in I}\vecone\{\SIGMA_i=\omega\}$.
Let $\bar X_{I,\omega}=\scal{X_{I,\omega}(\SIGMA)}{\nu}$. 
Then in order to establish (\ref{eqLemma_cutm1}) it suffices to prove that for any set $I\subset[n]$ and for any $\omega\in\Omega$,
	\begin{align}\label{eqLemma_cutm2}
	\scal{\vecone\cbc{\abs{X_{I,\omega}(\SIGMA)-\bar X_{I,\omega}}>\eps n/3}}{\nu}&<\eps/6&\mbox{and}\\
	\scal{\vecone\cbc{\abs{X_{I,\omega}(\SIGMA)-\bar X_{I,\omega}}>\eps n/3}}{\bar\nu}&<\eps/6.\label{eqLemma_cutm3}
	\end{align}
To prove (\ref{eqLemma_cutm2}) we estimate the second moment of $X_{I,\omega}$.
Due to $\delta$-symmetry we obtain
	\begin{align*}
	\scal{X_{I,\omega}(\SIGMA)^2}\nu&=\sum_{i,j\in I}\nu_{i,j}(\omega,\omega)\leq2\delta n^2+\bar X_{I,\omega}(1+\bar X_{I,\omega}).
	\end{align*}
Thus, Chebyshev's inequality yields
	\begin{align*}
	\scal{\vecone\cbc{\abs{X_{I,\omega}(\SIGMA)-\bar X_{I,\omega}}>\eps n/6}}{\nu}\leq\frac{2\delta n^2+\bar X_{I,\omega}}{(\eps n/6)^2}<\eps/6,
	\end{align*}
whence (\ref{eqLemma_cutm2}) follows.
Similarly, because $\bar\nu$ is a product measure we have $\scal{X_{I,\omega}(\SIGMA)^2}{\bar\nu}\leq\bar X_{I,\omega}(1+\bar X_{I,\omega})$
and thus another application of Chebyshev's inequality yields (\ref{eqLemma_cutm3}).
\end{proof}

\noindent
As a further preparation for the proof of \Thm~\ref{Thm_pin} we need the following extension of the triangle inequality.

\begin{lemma}\label{Lemma_xtriangle}
Suppose that $\mu_1,\nu_1,\ldots,\mu_k,\nu_k\in\cP(\Omega^n)$ and $p=(p_1,\ldots,p_k)\in\cP([k])$.
Let $\mu=\sum_{i=1}^kp_i\mu_i$, $\nu=\sum_{i=1}^kp_i\nu_i$.
Then
	\begin{align*}
	\Cutm(\mu,\nu)&\leq\sum_{i=1}^kp_i\Cutm(\mu_i,\nu_i).
	\end{align*}
\end{lemma}
\begin{proof}
Let $\gamma_i$ be a coupling of $\mu_i,\nu_i$ for which the cut distance is attained ($i=1,\ldots,k$).
Then we obtain a coupling $\gamma$ of $\mu,\nu$ by letting $\gamma(\sigma,\tau)=\sum_{i=1}^kp_i\gamma_i(\sigma,\tau)$.
With respect to this coupling we obtain for any $B\subset\Omega^n\times\Omega^n$, $I\subset[n]$, $\omega\in\Omega$ the bound
	\begin{align*}
	\frac1n\abs{\sum_{(\sigma,\tau)\in B}\gamma(\sigma,\tau)\sum_{i\in I}\vecone\{\sigma_i=\omega\}-\vecone\{\tau_i=\omega\}}
		&\leq\frac1n\sum_{i=1}^kp_i\abs{\sum_{(\sigma,\tau)\in B}\gamma_i(\sigma,\tau)\sum_{i\in I}\vecone\{\sigma_i=\omega\}-\vecone\{\tau_i=\omega\}}
			\leq\sum_{i=1}^kp_i\Cutm(\mu_i,\nu_i).
	\end{align*}
Since this bound holds for all $B,I,\omega$, the assertion follows.
\end{proof}

\begin{proof}[Proof of \Thm~\ref{Thm_pin}]
We may assume that $\eps<\eps_0$ is sufficiently small and pick $\delta=\delta(\eps,\Omega)=(\eps/(2\ln|\Omega|))^{c'}$ for a large enough constant $c'>0$.
By \Lem~\ref{Lemma_pinning} there exists a bounded random variable $\THETA\leq((2\ln|\Omega|)/\eps)^{c}$ such that 
	\begin{equation}\label{eqThm_pin1}
	\pr\brk{\mu^{\vI,\SIGMA}\mbox{ is $\delta$-symmetric}}>1-\delta.
	\end{equation}
By \Lem~\ref{Lemma_cutm} we can choose $c'>0$ so large that (\ref{eqThm_pin1}) implies
	$\Erw\brk{\Cutm(\mu^{\vI,\SIGMA},\bar\mu^{\vI,\SIGMA})}<\eps^5,$
whence the first part of \Thm~\ref{Thm_pin} is immediate.
Hence,
	\begin{align}\label{eqThm_pin2}
	\pr\brk{\Erw\brk{\Cutm(\mu^{\vI,\SIGMA},\bar\mu^{\vI,\SIGMA})\big|\vI}<\eps^2}\geq1-\eps^2.
	\end{align}
Furthermore,  for any set $\vI$  we can write $\mu$ as a convex combination
	$\mu=\sum_{\sigma\in\Omega^{\vI}}\mu_{\vI}(\sigma)\mu^{\vI,\sigma}.$ 
Therefore, (\ref{eqThm_pin2}) and \Lem~\ref{Lemma_xtriangle} yield
	\begin{align*}
	\Erw\brk{\Cutm(\mu(\sigma),\bar\mu^{\vI}(\sigma))}&\leq
			\Erw\brk{\Erw\brk{\Cutm\bc{\mu^{\vI,\SIGMA}(\sigma),\Erw\brk{\bar\mu^{\vI,\SIGMA}(\sigma)|\vI}}}}
		\leq2\eps^2,
	\end{align*}
as desired.
\end{proof}

\subsection{Proof of \Prop~\ref{Cor_symmetry}}
The first assertion follows from \Lem~\ref{Lemma_cutm} immediately.
We prove the second assertion by induction on $k$.
For $k=1$ there is nothing to show.
Thus, suppose $k>1$ and let $\bar\mu=\bigotimes_{i=1}^n\mu_i$.
Let $\delta=\eps^4/(128|\Omega|)^{4k}$ and assume that $\Cutm(\mu,\bar\mu)<\delta$.
For spins $\omega_1,\ldots,\omega_{k-1}\in\Omega$ and pairwise distinct $i_1,\ldots,i_{k-1}\in[n]^k$ let
	$$S_{i_1,\ldots,i_{k-1}}\bc{\omega_1,\ldots,\omega_{k-1}}=\{\sigma\in\Omega^n:\sigma_{i_1}=\omega_1,\ldots,\sigma_{i_{k-1}}=\omega_{k-1}\}.$$
Further, let
	\begin{align}\label{eqCor_symmetry_1}
	J(\omega_1,\ldots,\omega_{k-1})&=\cbc{
		(i_1,\ldots,i_{k-1}):\mu\bc{S_{i_1,\ldots,i_{k-1}}\bc{\omega_1,\ldots,\omega_{k-1}}}>\eps/(4|\Omega|^{k-1})}.
	\end{align}
Additionally, for a further spin $\omega_k\in\Omega$ let 
	\begin{align}\label{eqCor_symmetry_2}
	I_{i_1,\ldots,i_{k-1}}\bc{\omega_1,\ldots,\omega_k}&=\cbc{i_k\in[n]\setminus\{i_1,\ldots,i_{k-1}\}:
	\mu_{i_k}(\omega_k|S_{i_1,\ldots,i_{k-1}}(\omega_1,\ldots,\omega_{k-1}))\leq\mu_i(\omega_k)-\eps/(4|\Omega|)}.
	\end{align}	
We claim that
	\begin{align}\label{eqCor_symmetry1}
	|I_{i_1,\ldots,i_{k-1}}\bc{\omega_1,\ldots,\omega_k}|&\leq\eps n/(8|\Omega|^k)
	&\mbox{for all }i_1,\ldots,i_{k-1}\in J(\omega_1,\ldots,\omega_{k-1}).
	\end{align}

To establish (\ref{eqCor_symmetry1}) let $S=S_{i_1,\ldots,i_{k-1}}\bc{\omega_1,\ldots,\omega_{k-1}}$, $I=I_{i_1,\ldots,i_{k-1}}\bc{\omega_1,\ldots,\omega_k}$ for brevity and set
	$$T=S\cap\cbc{\sigma\in\Omega^n:\sum_{i\in I}\vecone\{\sigma_i=\omega_k\}\geq\sum_{i\in I}(\mu_i(\omega_k)-\eps/(8|\Omega|))}.$$
Then by Markov's inequality and (\ref{eqCor_symmetry_2}),
	\begin{align}\label{eqCor_symmetry2}
	\mu(T|S)&\leq\frac{1-\eps/(4|\Omega|)}{1-\eps/(8|\Omega|)}\leq
		1-\frac{\eps}{16|\Omega|}.
	\end{align}
On the other hand, if (\ref{eqCor_symmetry1}) is violated then Chebyshev's inequality applied to the product measure $\bar\mu$ shows that for large enough $n$
the event
	$$T'=\cbc{\sigma\in\Omega^n:\sum_{i\in I}\vecone\{\sigma_i=\omega_k\}\geq\sum_{i\in I}(\mu_i[\omega_k]-\eps/(128|\Omega|))}$$
satisfies
	\begin{align}\label{eqCor_symmetry3}
	\bar\mu(T')&\geq1-o(1).
	\end{align}
But (\ref{eqCor_symmetry2}) and (\ref{eqCor_symmetry3}) contradict the assumption that $\Cutm(\mu,\bar\mu)<\delta$
	(by plugging in the event $B=(S\setminus T)\times T'$ and the set $I$).
Thus, we have established (\ref{eqCor_symmetry1}).

Now, by the induction hypothesis we have
\begin{align}\nonumber
\sum_{i_1<\cdots<i_k}\TV{\mu_{i_1,\ldots,i_k}-\bigotimes_{j=1}^k\mu_{i_j}}
		&\leq\sum_{i_1<\cdots<i_k}\TV{\mu_{i_1,\ldots,i_k}-\mu_{i_1,\ldots,i_{k-1}}\tensor\mu_{i_k}}+
			\TV{\mu_{i_1,\ldots,i_{k-1}}\tensor\mu_{i_k}-\bigotimes_{j=1}^k\mu_{i_j}}\\
		&\leq\frac{\eps n^k}4+\sum_{i_1<\cdots<i_k}\TV{\mu_{i_1,\ldots,i_k}-\mu_{i_1,\ldots,i_{k-1}}\tensor\mu_{i_k}}.
			\label{eqCor_symmetry_3}
\end{align}
Further, due to (\ref{eqCor_symmetry_1}), (\ref{eqCor_symmetry_2}) and (\ref{eqCor_symmetry1}),
	\begin{align*}
	\sum_{i_1<\cdots<i_k}&\TV{\mu_{i_1,\ldots,i_k}-\mu_{i_1,\ldots,i_{k-1}}\tensor\mu_{i_k}}=\frac12\sum_{\omega_1,\ldots,\omega_k\in\Omega}
		\sum_{i_1<\cdots<i_k}\abs{\mu_{i_1,\ldots,i_k}(\omega_1,\ldots,\omega_k)-\mu_{i_1,\ldots,i_{k-1}}(\omega_1,\ldots,\omega_{k-1})\tensor\mu_{i_k}(\omega_k)}\\
		&\leq\frac{\eps n^k}4+\sum_{\omega_1,\ldots,\omega_k\in\Omega}\sum_{\substack{(i_1,\ldots,i_{k-1})\\\in J(\omega_1,\ldots,\omega_{k-1})}}
			\sum_{i_k>i_{k-1}}\max\{0,\mu_{i_1,\ldots,i_{k-1}}(\omega_1,\ldots,\omega_{k-1})\mu_{i_k}(\omega_k)
				-\mu_{i_1,\ldots,i_k}(\omega_1,\ldots,\omega_k)\}\\
		&\leq\frac{\eps n^k}2+\sum_{\omega_1,\ldots,\omega_k\in\Omega}\sum_{\substack{(i_1,\ldots,i_{k-1})\\\in J(\omega_1,\ldots,\omega_{k-1})}}
			|I_{i_1,\ldots,i_{k-1}}\bc{\omega_1,\ldots,\omega_k}|\leq\frac{5\eps n^k}8.
	\end{align*}
Thus, the assertion follows from (\ref{eqCor_symmetry_3}).

\subsection{Proof of \Prop~\ref{Cor_states}}
\Lem~\ref{Lemma_cutm} implies that
	\begin{align}\label{eqCor_states1}
	\Cutm\bc{\mu[\nix|S_i],\bigotimes_{j=1}^n\mu[\nix|S_i]}&<\eps\qquad\mbox{for }i=1,\ldots,l.
	\end{align}
Moreover, we have the trivial bound $\Cutm(\mu[\nix|S_0],\bigotimes_{j=1}^n\mu[\nix|S_0])\leq1$.
Since $\mu(S_0)<\eps$, the assertion follows from (\ref{eqCor_states1}) and \Lem~\ref{Lemma_xtriangle}.

\subsection{Proof of \Cor~\ref{Cor_pin}}
We begin with the following statement regarding product measures.

\begin{lemma}\label{Lemma_prod}
For any $\Omega$, $\eps>0$, $k\geq2$ there exist $\delta>0$ and $n_0>0$ such that for all $n>n_0$ and all $\mu,\mu',\nu,\nu'\in\cP(\Omega^n)$ the following is true.
If $\Cutm(\mu,\mu')+\Cutm(\nu,\nu')<\delta$, then $\Cutm(\mu\tensor\nu,\mu'\tensor\nu')<\eps$.
\end{lemma}
\begin{proof}
Due to the triangle inequality it suffices to prove the assertion under the additional assumption that $\nu=\nu'$.

Let $\gamma$ be a coupling of $\mu,\mu'$ for which the cut distance is attained.
We define a coupling $g$ of $\mu\tensor{}\nu,\mu'\tensor{}\nu$ by letting $g(\sigma,\tau,\sigma',\tau)=\gamma(\sigma,\sigma')\nu(\tau)$.
Then for any $I\subset[n]$, $B\subset \Omega^{4n},\omega,\omega'\in\Omega$ we have
\begin{align*}
&\abs{\sum_{x\in I}\sum_{(\sigma,\tau,\sigma',\tau)\in B}g(\sigma,\tau,\sigma',\tau)\vecone\cbc{\tau_x=\omega'}\bc{\vecone\cbc{\sigma_x=\omega}-\vecone\cbc{\sigma_x'=\omega}}}\\
&\qquad\leq\sum_{\tau\in\Omega^n}\nu(\tau)\abs{\sum_{x\in I:\tau_x=\omega'}
		\sum_{\sigma,\sigma'\in\Omega^n:(\sigma,\tau,\sigma',\tau)\in B}\gamma(\sigma,\sigma')\bc{\vecone\cbc{\sigma_x=\omega}-\vecone\cbc{\sigma_x'=\omega}}}\leq \Cutm(\mu,\mu').
\end{align*}
Since this bound holds for all $I,B,\omega,\omega'$, we conclude that $\Cutm(\mu\tensor{}\nu,\mu'\tensor{}\nu)\leq\Cutm(\mu,\mu')$.
\end{proof}

\begin{proof}[Proof of \Cor~\ref{Cor_pin}]
Due to \Lem~\ref{Lemma_prod} it suffices to deal with the case $k=1$.
Given $\eps>0$ we choose small enough $\xi=\xi(\eps)>0$, $\zeta=\zeta(\xi)>0$ and $\delta=\delta(\zeta)>0$ and assume that $\Cutm(\mu,\nu)<\delta$.
Let $\gamma$ be a coupling of $\mu,\nu$ for which the cut distance is attained.
Invoking \Thm~\ref{Thm_pin} and \Prop~\ref{Cor_symmetry}, we obtain partitions $S_0,\ldots,S_k$, $T_0,\ldots,T_l$ of $\Omega^n$ such that
$\mu(S_0),\nu(T_0)<\xi$, $\mu(S_i),\nu(T_i)>\zeta$ for all $i\geq1$ and such that the distributions $\mu[\nix|S_i]$ and $\nu[\nix|T_i]$ are $\xi$-symmetric for all $i\geq1$.
Let $v_i=\mu(S_i)$ for $i=0,\ldots,k$ and $w_i=\nu'(T_i)$ for $i=0,\ldots,l$.
The coupling $\gamma$ induces a coupling $g$ of the two distributions $(v_i)_i$ and $(w_i)_i$.

We claim that
	\begin{align}\label{eqCor_pin1}
	\frac1n\sum_{i\in[k],j\in[l]}\sum_{h=1}^ng(i,j)\TV{\mu_h[\nix|S_i]-\nu_h[\nix|T_j]}&<\xi^{1/9}.
	\end{align}
To see this, define for $i\in[k],\ j\in[l],\ \omega\in\Omega$,
\begin{align*}
I_{i,j}(\omega)&=\cbc{x\in[n]:\mu_x(\omega|S_i)>\nu_x(\omega|T_j)+\xi^{1/8}},\qquad
X_{\omega,i,j}(\sigma)=\sum_{x\in I_{i,j}(\omega)}\vecone\cbc{\sigma_x=\omega}\qquad(\sigma\in\Omega^n),\\
\cE_{i,j}(\omega)&=\cbc{(\sigma,\tau)\in S_i\times T_j:
		\abs{X_{\omega,i,j}(\sigma)-\scal{X_{\omega,i,j}}{\mu(\nix|S_i)}}<\xi^{1/4}n,\ 
		\abs{X_{\omega,i,j}(\tau)-\scal{X_{\omega,i,j}}{\nu(\nix|T_j)}}<\xi^{1/4}n},\\
\cF_{i,j}(\omega)&=\cbc{(\sigma,\tau)\in S_i\times T_j:\abs{X_{\omega,i,j}(\sigma)-X_{\omega,i,j}(\tau)}<\xi n}.
\end{align*}
Due to $\xi$-symmetry we have for all $i,j\geq1$,
\begin{align*}
\scal{X_{\omega,i,j}^2}{\mu(\nix|S_i)}-\scal{X_{\omega,i,j}}{\mu(\nix|S_i)}^2&\leq 2\xi n^2,&
\scal{X_{\omega,i,j}^2}{\nu(\nix|T_j)}-\scal{X_{\omega,i,j}}{\nu(\nix|T_i)}^2&\leq 2\xi n^2.
\end{align*}
Hence, Chebyshev's inequality implies that for all $\omega\in\Omega$,
\begin{align}\label{eqCor_pin2_a}
\sum_{i\in[k],j\in[l]}\gamma(\cE_{i,j}(\omega))\geq1-\xi^{1/5}.
\end{align}
Further, assuming that $\Cutm(\mu,\nu)<\delta$ for a small enough $\delta>0$, we see that
\begin{align}\label{eqCor_pin2_b}
\sum_{i\in[k],j\in[l]}\gamma(\cF_{i,j}(\omega))\geq1-\xi.
\end{align}
Combining \eqref{eqCor_pin2_a}--\eqref{eqCor_pin2_b}, we obtain
\begin{align}\label{eqCor_pin2_c}
\sum_{i\in[k],j\in[l]}\gamma(\cE_{i,j}(\omega)\cap\cF_{i,j}(\omega))\ge1-2\xi^{1/5}\qquad\mbox{ for all }\omega\in\Omega.
\end{align}
Now, if $\gamma(\cE_{i,j}(\omega)\cap\cF_{i,j}(\omega))>0$, then  by the triangle inequality,
\begin{align*}
\xi^{1/8}|I_{i,j}(\omega)|\leq\sum_{x\in I_{i,j}(\omega)}\mu_x(\omega|S_i)-\nu_x(\omega|T_j)\leq
\abs{\scal{X_{\omega,i,j}}{\mu(\nix|S_i)}-\scal{X_{\omega,i,j}}{\nu(\nix|T_j)}}\leq3\xi^{1/4}n,
\end{align*}
and thus $|I_{i,j}(\omega)|\leq3\xi^{1/8}n$.
Consequently, \eqref{eqCor_pin2_c} implies that $\sum_{i\in[k],j\in[l]}g(i,j)|I_{i,j}(\omega)|\leq4\xi^{1/8}n$ for all $\omega\in\Omega$,
whence we obtain \eqref{eqCor_pin1}.

To complete the proof we may assume without loss that $f(\sigma)=\sum_{i_1\in I_1,\ldots,i_\ell\in I_\ell}\vecone\{\sigma_{i_1}=\tau_1,\ldots,\sigma_{i_\ell}=\tau_\ell\}.$
Then
	\begin{align}
	\abs{\scal{f(\SIGMA)}\mu-\scal{f(\SIGMA)}\nu}
		&\leq4\xi+\sum_{i\in[k],j\in[l]}g(i,j)\abs{\scal{f(\SIGMA)}{\mu[\nix|S_i]}-\scal{f(\SIGMA)}{\nu[\nix|T_j]}}.
			\label{eqCor_pin2}
	\end{align}
Furthermore, 
since the conditional distributions $\mu(\nix|S_i)$ and $\nu(\nix|T_j)$ are $\xi$-symmetric for all $i\in[k]$, $j\in[l]$, \Prop~\ref{Cor_symmetry} yields
	\begin{align}			\label{eqCor_pin3}
	\abs{\scal{f(\SIGMA)}{\mu[\nix|S_i]}-\scal{f(\SIGMA)}{\nu[\nix|T_j]}}&\leq\frac{\eps n^l}2+\sum_{i_1\in I_1,\ldots,i_l\in I_l}
		\abs{\prod_{h=1}^l\mu_{i_h}[\tau_h|S_i]-\prod_{h=1}^l\nu_{i_h}[\tau_h|T_j]}&&(i\in[k],j\in[l]).
	\end{align}
Finally, the assertion follows from (\ref{eqCor_pin1}), (\ref{eqCor_pin2}) and (\ref{eqCor_pin3}). 
\end{proof}

\subsection{Proof of \Cor~\ref{Cor_overlap}}
Choose small enough $\zeta=\zeta(\xi)>0$ and $\delta=\delta(\zeta)>0$ and assume that $\Cutm(\mu,\nu)<\delta$.
Then \Lem~\ref{Lemma_prod} shows that $\Cutm(\mu\tensor{}\mu,\nu\tensor{}\nu)<\zeta$.
Hence, there exists a coupling $\gamma\in\cP(\Omega^{4n})$ of $\mu\tensor{}\mu$ and $\nu\tensor{}\nu$ such that for all $B\subset\Omega^{4n}$ and all $\omega,\omega'\in\Omega$ we have
\begin{align*}\nonumber
\zeta&>\frac1n\abs{\sum_{x=1}^n\sum_{(\sigma,\tau,\sigma',\tau')\in B}\gamma(\sigma,\tau,\sigma',\tau')
			\bc{\vecone\cbc{\sigma_x=\omega,\tau_x=\omega'}-\vecone\cbc{\sigma_x'=\omega,\tau_x'=\omega'}}}\\
&=\abs{\sum_{(\sigma,\tau,\sigma',\tau')\in B}\gamma(\sigma,\tau,\sigma',\tau')\bc{\rho_{\sigma,\tau}(\omega,\omega')-\rho_{\sigma',\tau'}(\omega,\omega')}}.
\end{align*}
Consequently, the events
\begin{align*}
B_\eps(\omega,\omega')&=\cbc{(\sigma,\tau,\sigma',\tau')\in\Omega^{4n}:
			\rho_{\sigma,\tau}(\omega,\omega')>\rho_{\sigma',\tau'}(\omega,\omega')+\eps/(|\Omega|^2)}
\end{align*}
have probability $\gamma(B_\eps(\omega,\omega'))<\zeta^{1/2}$.
In effect, with $(\SIGMA,\TAU,\SIGMA',\TAU')$ chosen from $\gamma$ we have
\begin{align*}
\Erw\TV{\rho_{\SIGMA,\TAU}-\rho_{\SIGMA',\TAU'}}&=\frac{1}{2}\sum_{\omega,\omega'\in\Omega}\Erw\abs{\rho_{\SIGMA,\TAU}(\omega,\omega')-\rho_{\SIGMA',\TAU'}(\omega,\omega')}\leq\frac\eps2+\sum_{\omega,\omega'\in\Omega}\gamma(B_\eps(\omega,\omega'))<\eps,
\end{align*}
as desired.

\section{Bethe states}\label{Sec_BetheStates}

\noindent
In this section we establish the main result of the paper (\Thm~\ref{Thm_BetheStates} below) on the Bethe state decomposition for a broad class of models.
\Thm~\ref{thmPotts} regarding the Potts model is an immediate application of this more general result.

\subsection{Factor graph models}\label{Sec_factorGraphs}
Factor graphs provide a suitably general framework to discuss the Bethe state decomposition.
Let $\Omega\neq\emptyset$ be a finite set of spins and let $k\geq2$ be an integer.
An {\em $\Omega$-factor graph} $G=(V,F,(\partial a)_{a\in F},(\psi_a)_{a\in F})$ consists of a set $V$ of {\em variable nodes}, a set $F$ of {\em constraint nodes} and a family $\partial a=(\partial_1 a,\ldots,\partial_{k} a)\in V^{k}$ of {\em neighbors} as well as a {\em weight function} $\psi_a:\Omega^{k}\to(0,\infty)$ for each $a\in F$. 
The {\em total weight} of a configuration $\sigma\in\Omega^V$ is defined as
	\begin{align}\label{eqweight}
	\psi_G(\sigma)&=\prod_{a\in F}\psi_a(\sigma(\partial_1 a),\ldots,\sigma(\partial_k a)).
	\end{align}
These weights naturally give rise to the {\em Gibbs distribution} of $G$, defined by
	\begin{align}\label{eqGibbs}
	\mu_G(\sigma)&=\psi_G(\sigma)/Z(G),\qquad\mbox{where}\qquad Z(G)=\sum_{\tau\in\Omega^V}\psi_G(\tau).
	\end{align}
For a subset $U\subset V$ we denote by $\mu_{G,U}$ the joint marginal distribution of the variables in $U$.

As a generalization of models such as the Potts model on the \Erdos-\Renyi\ graph we investigate the following class of random factor graphs.
Let $\Psi$ be a finite set of weight functions $\Omega^k\to(0,\infty)$ and let $P$ be a probability distribution on $\Psi$.
Then for integers $n,m$ we let $\G(n,m,P)$ be the random factor with variable nodes $V=V_n=\{x_1,\ldots,x_n\}$ and constraint nodes $F=F_m=\{a_1,\ldots,a_m\}$ where, independently for each $i=1,\ldots,m$, $\partial a_i\in V^k$ is chosen uniformly and independently of $\psi_{a_i}\in\Psi$, which is chosen from $P$.
Additionally, for $d>0$ we let $\vm=\vm_{d,n}$ be a random variable with distribution $\Po(dn/k)$ and we introduce the shorthand $\G=\G(n,\vm,P)$.
The random factor graph $\G$ has a property $\cE$ {\em asymptotically almost surely} (`\whp') if $\lim_{n\to\infty}\pr\brk{\G\in\cE}=1$.
This model, and special cases thereof, has received a great deal of attention over the past few years (e.g.,~\cite{Victor,CKPZ,CDGS,DemboBrazil,demboPotts,MM}).

\begin{example}
The Potts antiferromagnet on the \Erdos-\Renyi\ random graph can be expressed easily as a random factor graph model on the spin set $\Omega=\{1,\ldots,q\}$, $k=2$ and the single weight function $\psi_\beta(\sigma_1,\sigma_2)=\exp(-\beta\vecone\{\sigma_1=\sigma_2\})$.
Thus, $\Psi=\{\psi_\beta\}$ and $P$ is the atom on $\psi_\beta$, of course.
Standard arguments show that this model and the Potts model on the \Erdos-\Renyi\ graph are mutually contiguous.
\end{example}

\begin{example} 
The random $k$-SAT model is another example~\cite{PanchenkoSAT,TalagrandSAT}. 
Here $\Omega = \{ \pm 1\}$ and  $\Psi$ consists of $2^k$ different constraint functions indexed by $\xi \in \{\pm1\}^k$ with 
	\begin{align*}
	\psi_{\xi}(\sigma) &= 1-\frac{1-\eul^{-\beta}}{2^k} \prod_{i=1}^k\bc{1+\xi_i\sigma_i}.
	\end{align*}
The distribution $P$ is uniform on $\Psi$.   Each constraint node with constraint function $\psi_{\xi}$ corresponds to a Boolean $k$-clause with literal signs given by $\xi$.  The Gibbs measure $\mu_{\G}$ on Boolean assignments $\sigma \in \{\pm 1\}^n$ assigns probability proportional to $\exp(-\beta H(\G,\sigma ))$ where $H(\G,\sigma)$ is the number of unsatisfied clauses under the assignment $\sigma$. 
\end{example}

\subsection{Bethe states}\label{Sec_introBethe}
To state our main result we need to formally introduce Bethe states, generalizing the notion presented in Theorem~\ref{thmPotts}.
A factor graph $G$ induces a bipartite graph on the sets $V,F$ of variable and constraint nodes where we connect each $a\in F$ with the variable nodes $\partial_1a,\ldots,\partial_ka$.
This allows us to use graph-theoretic concepts such as the distance between two nodes (viz.\ the length of a shortest path).
In particular, we say that $a\in F$ and $x\in V$ are {\em adjacent} if these two nodes are connected by an edge in the bipartite graph. We write $\partial x$ for the set of neighbors of a variable node $x$.
Moreover, while $\partial a$ really is an ordered $k$-tuple of for each constraint node $a$, use the same symbol to denote the set $\{\partial_1a,\ldots,\partial_ka\}$ of neighbors of $a$.

Further, a subset $\emptyset\neq U \subset V$ is a \textit{cavity} of $G$ if
\begin{description}
\item[CAV1] the induced subgraph $G[U]$, consisting of variable nodes $U$ and all constraint nodes $a \in F$ so that $\partial a \subseteq U$, is acyclic, and
\item[CAV2] for any $a \in F$ so that $\partial a \not \subseteq U$, $|\partial a \cap U|\le 1$, and
\item[CAV3] for any $a,b \in F$, $a \ne b$, so that $\partial a \cap U \ne \emptyset$ and  $\partial b \cap U \ne \emptyset$, $\partial a \cap \partial b \setminus U = \emptyset$. 
\end{description}
See Figure~\ref{figCavity} for an illustration.
Let $\partial U$ be the set of constraint nodes $a$ so that $\partial a \cap U \ne \emptyset$ and $\partial a \setminus U \ne \emptyset$.
For $a \in \partial U$, let $x_{a\to U}$ denote the unique variable node $x\in U \cap \partial a$.
Finally, let $\mathcal C(G,l,r)$ be the set of all cavities $U$ of size $|U | = l$ such that $G[U]$ has precisely $r$ connected components.

\begin{figure}
\centering
\includegraphics[scale=0.35]{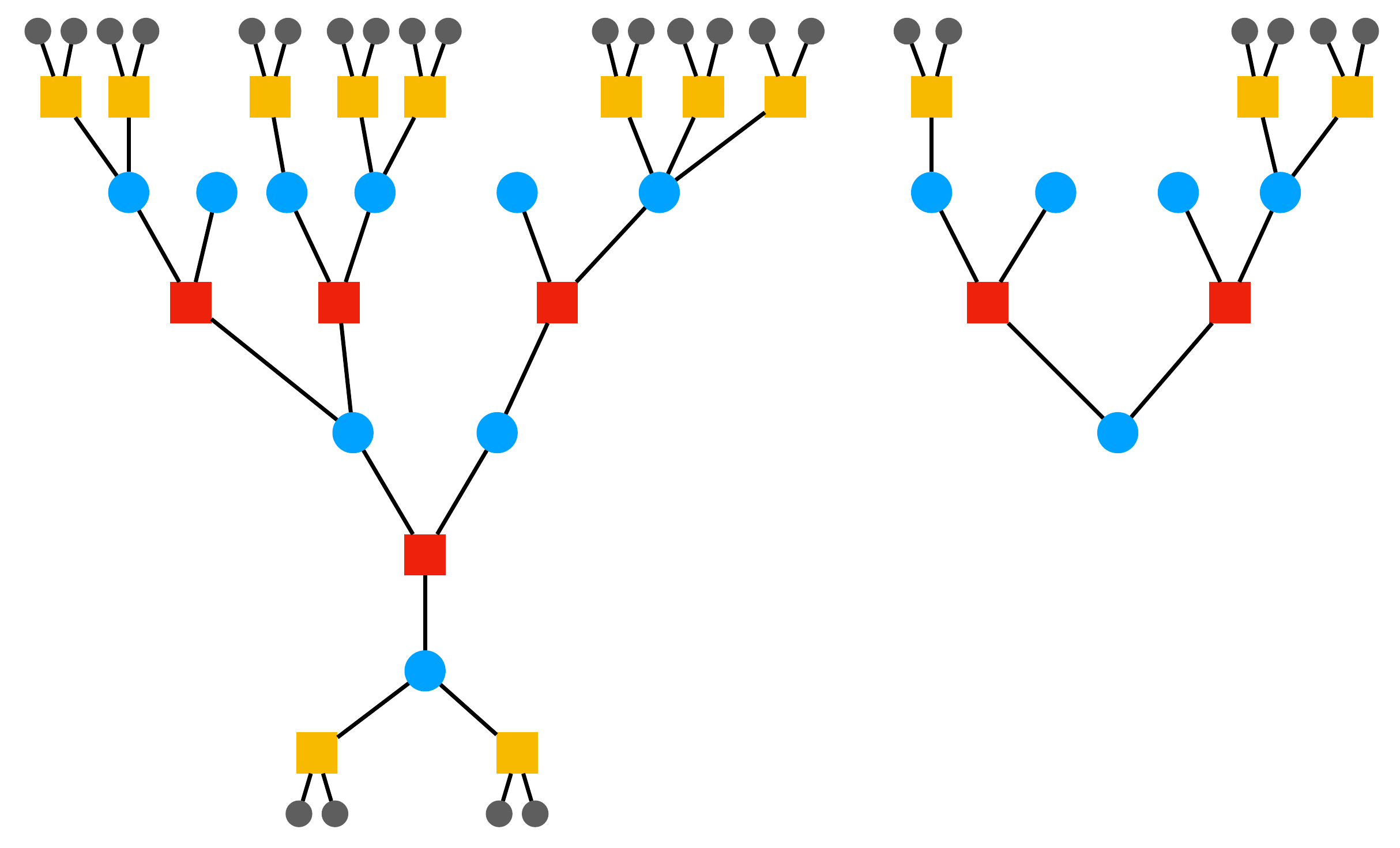}
\caption{A cavity $U$ of size $14$ with two connected components. The variable nodes of the cavity are in blue.  The induced graph $G[U]$ is given by the blue variable nodes and red constraint nodes.  The boundary constraint nodes, $\partial U$, are in gold.  The gray nodes are variables nodes at distance $2$ from the cavity.}
\label{figCavity}
\end{figure}

We would like to express that in a Bethe state, for `almost all' small cavities $U$ the joint distribution $\mu_{G, U}$ of the variables in $ U$ is given by the internal Gibbs measure $\mu_{G[U]}$ of the cavity times an `independent boundary condition'. The internal distribution of the cavity $U$ is simply the Gibbs measure on the induced subgraph $G[U]$,
\begin{align*}
\mu_{G[U]}[\sigma] & \propto \prod_{a \in G[U]} \psi_a(\sigma|_{\partial a}) \, ,
\end{align*}
and this distribution is particularly easy to understand as $G[U]$ is acyclic. 
To define the boundary condition, suppose that $\emptyset\neq S\subset\Omega^V$ is a set of configurations, that $x$ is a  variable node and that $a\in\partial x$.
Then we let
\begin{align}\label{eqmsg2}
	\mu_{G,a\to x}[\tau|S]&\propto\sum_{\sigma\in S}\vecone\{\sigma_x=\tau\}
		\prod_{b\in F\setminus(\partial x\setminus a)}\psi_b(\sigma|_{\partial b})\qquad(\tau\in\Omega)
	\end{align}
be the Gibbs marginal of $x$ given $S$ in the factor graph obtained from $G$ by removing all constraint nodes $b\in\partial x$ except $a$. We call $\mu_{G,a\to x}[\nix | S]$ the \textit{standard message from $a$ to $x$  given $S$}. 

With this definition, and recalling that for $a \in \partial U$, $x_{a\to U}$ is the unique $x \in U\cap \partial a$, the aforementioned ``internal times independent boundary'' distribution reads
	\begin{align}\label{eqMubar}
	\bar\mu_{G,U}[\sigma |S] & \propto  \mu_{G[U]}[\sigma ]  \prod_{a \in \partial U} \mu_{G,a\to x}[\sigma_{x_{a\to U}}|S]\,.
	\end{align}
Guided by~\cite[\Chap~19]{MM}, we call $S$ an {\em $(\eps,\ell)$-Bethe state} 
 of $G$ if for all $1\leq r\leq l\leq\ell$,	
	\begin{align}\label{eqBetheState}
	\sum_{U\in\cC(G,l,r)}\sum_{\sigma\in \Omega^{ U}}\abs{\mu_{G, U}[\sigma|S]
		-\bar\mu_{G,U}[\sigma|S]}
				&<\eps\cdot|\cC(G,l,r)|.
	\end{align}
In words, for almost all cavities the actual induced Gibbs distribution $\mu_{G, U}[\sigma|S]$ is close in total variation distance to the idealized distribution from (\ref{eqMubar}).

\begin{theorem}\label{Thm_BetheStates}
For any $k\geq2$, $\Omega\neq\emptyset$, $\Psi$ and $d>0$ and for any sequence $L=L(n)\to\infty$ there
exist $\eps=\eps(n)\to0$, $K=K(n)\to\infty$ such that the random factor graph $\G$ has the following property \whp\
There exist pairwise disjoint non-empty sets $S_1=S_1(\G),\ldots,S_l=S_l(\G)\subset\Omega^{V_n}$ with $1\leq l\leq L(n)$ such that
	\begin{enumerate}[(i)]
	\item $\sum_{j=1}^l\mu_{\G}(S_j)\geq1-\eps$,
	\item For each $j=1, \ldots l$, $S_j$ is an $(\eps,\ell)$-Bethe state of $\G$ for each $\ell=1,\ldots,K$.
	\end{enumerate}
\end{theorem}

\noindent
Thus, \whp\ the random factor graph $\G$ admits a decomposition of $\Omega^{V_n}$  into Bethe states $S_1,\ldots,S_l$ that cover almost the entire probability mass.
Moreover, choosing a suitable $L=L(n)$ we can ensure that the total number of Bethe states in this decomposition diverges only very slowly.

The definition (\ref{eqBetheState}) of a Bethe state suits the situation of random factor graph models very well.
Indeed, the average variable node degree in $\G$ is asymptotically equal to $d$, and thus bounded independently of $n$.
Hence, the subgraph of $\G$ within a bounded distance of all but a bounded number of variable nodes is of bounded size and acyclic and thus a cavity \whp\
Consequently, in a random factor graph $\G$ the condition (\ref{eqBetheState}) characterizes both short-range and long-range correlations.
Specifically, to characterize the distribution induced by $\mu_{\G}$ in the vicinity of a variable node $x$ we just let $U$ be a bounded-depth neighborhood of $x$ in $\G$, as in Theorem~\ref{thmPotts}.
With respect to long-range correlations, if we choose two variable nodes $x,y$ randomly, then typically their distance in $\G$ will be as large as $\Omega(\ln n)$ (because $\G$ is of bounded average degree).
Thus, due to the product structure of (\ref{eqMubar}) over the connected components of a cavity, (\ref{eqBetheState}) shows that bounded-size cavities around $x,y$ are essentially uncorrelated.
Therefore, in perfect analogy to (\ref{eqRSpotts}) we obtain from \Thm~\ref{Thm_BetheStates} immediately that
	\begin{align}\label{eqRSBG}
	\lim_{n\to\infty}\frac1{n^2}\sum_{x,y\in V_n, x\neq y}\Erw\brk{\sum_{i=1}^l\mu_{\G}(S_i)\TV{\mu_{\G,x,y}[\nix|S_i]
		-\mu_{\G,x}[\nix|S_i]\tensor \mu_{\G,y}[\nix|S_i]}}&=0.
	\end{align}

Clearly, in order to actually use the Bethe state decomposition we need to evaluate the expression (\ref{eqMubar}).
The Gibbs distribution of the cavity $G[U]$ itself is easy to cope with because $G[U]$ is acyclic.
However, it is less clear how to get a handle on the messages $\mu_{G,a\to x}[\nix|S]$.
A key hypothesis of the cavity method is that these quantities are (approximate) fixed points of a certain non-linear operator that can be written down explicitly in term of just the graph $G$, the {\em Belief Propagation operator}~\cite[\Chap~19]{MM}.

Our second main result verifies this hypothesis.
To state the result, we define the {\em message space} $\cM(G)$ of a factor graph $G=(V,E,(\partial a)_{a\in F},(\psi_a)_{a\in F})$ as the set of all families $\nu=(\nu_{x\to a},\nu_{a\to x})_{x\in V,a\in F,x\in\partial a}$ of probability distributions $\nu_{x\to a},\nu_{a\to x}\in\cP(\Omega)$.
We define a metric $\cD_1$ on $\cM(G)$ by letting
	\begin{align*}
	\cD_1(\nu,\nu')&=\frac1{|V|}\sum_{x\in V,a\in\partial x}
		\TV{\nu_{x\to a}-\nu'_{x\to a}}+\TV{\nu_{a\to x}-\nu'_{a\to x}}.
	\end{align*}
The {\em Belief Propagation operator} of $G$ is the map $\BP:\cM(G)\to\cM(G)$, $\nu\mapsto\hat\nu=\BP(\nu)$ defined by
	\begin{align*}
	\hat\nu_{x\to a}(\sigma)&\propto\prod_{b\in\partial x\setminus a}\nu_{b\to x}(\sigma),&
	\hat\nu_{a\to x}(\sigma)&\propto\sum_{\tau\in\Omega^{\partial a}}\vecone\{\tau_x=\sigma\} \psi_a(\tau|_{\partial a})
		\prod_{y\in\partial a\setminus x}\nu_{y\to a}(\tau_y).
	\end{align*}
We call $\nu\in\cM(G)$ an {\em $\eps$-Belief Propagation fixed point} if $\cD_1(\nu,\BP(\nu))<\eps.$

Finally, having defined the messages $\mu_{G,a\to x}[\nix|S]$ in (\ref{eqmsg2}) already, we define the {\em standard message from $x$ to $a$} by
	\begin{align}\label{eqmsg1}
	\mu_{G,x\to a}[\tau|S]&\propto\sum_{\sigma\in S}\vecone\{\sigma_x=\tau\}\prod_{b\in F\setminus a}
		\psi_b(\sigma|_{\partial_b}).
		&(\tau\in\Omega)
	\end{align}
Thus, $\mu_{G,x\to a}[\tau|S]$ is the Gibbs marginal of $x$ given $S$ in the factor graph obtained from $G$ by deleting the constraint node $a$.

\begin{corollary}\label{Cor_BP}
Under the assumptions of \Thm~\ref{Thm_BetheStates} \whp\ there exists a Bethe state decomposition $S_1,\ldots,S_l$ with the additional property:
	\begin{enumerate}[(i)]
	\item[(iii)] the set of canonical messages $(\mu_{\G,x\to a}[\nix|S_i],\mu_{\G,a\to x}[\nix|S_i])_{x\in V_n,a\in\partial x}$
		is an $\eps$-Belief Propagation fixed point  for each $i=1,\ldots,l$.
	\end{enumerate}
\end{corollary}

Finally, we observe that the partition given in Theorem~\ref{Thm_BetheStates} is essentially unique 
in the sense that any two such decompositions render approximations of the Gibbs distribution of $\G$ that are close under the cut metric.

\begin{corollary}\label{Cor_BetheCutMetric}
 For every $\eps' >0$ there is $\eps >0$ so that the following is true. Let $S_1, \dots S_l$ 
 a partition of $\Omega^{V_n}$ 
 satisfies the statements (i) and (ii) from \Thm~\ref{Thm_BetheStates}.
 Let $z=\sum_{j=1}^l\mu_{\G}(S_j)$.
 Then
 	\begin{align*}
	{\Cutm\bc{\mu_{\G},\frac1z\sum_{j=1}^l\mu_{\G}(S_j)\bigotimes_{i=1}^n\mu_{\G,x_i}[\nix|S_j]}}=o(1).
	\end{align*}
\end{corollary}

\Thm~\ref{Thm_BetheStates} and \Cor~\ref{Cor_BP} contribute to the conceptual vindication of the replica symmetry breaking hypothesis.
In particular, they establish for the first time in a generic, universal way a connection between the fixed points of the Belief Propagation operator and the actual Gibbs measure of a random factor graph.
Specifically, (\ref{eqRSBG}) shows the absence of extensive long-range correlations within Bethe states, (\ref{eqBetheState}) characterizes the conditional distribution within cavities in terms of messages and \Cor~\ref{Cor_BP} shows that the canonical messages of the Bethe states are approximate Belief Propagation fixed points.

\medskip
\paragraph{\em Discussion and related work.}

While the existence of a Bethe state decomposition is a universal feature of random factor graphs, the specific Bethe states depend on  the model.
Indeed, as the example of the Potts antiferromagnet shows the Bethe state decomposition may undergo phase transitions as the parameter values of the model change.
For instance, for some models/parameter values  there may be a decomposition consisting of just a single Bethe state.
This is called the {\em replica symmetric} case in physics jargon.
The Potts antiferromagnet on the \Erdos-\Renyi\ graph with average degree $d$ below the condensation threshold $\dc(q,\beta)$ is an example of a replica symmetric model.
Indeed, by generalization of (\ref{eqRS}) replica symmetric models are characterized by the condition
	\begin{align}\label{eqRSG}
	\lim_{n\to\infty}\frac1{n^2}\sum_{x,y\in V_n,x\neq y}\Erw\TV{\mu_{\G,x,y}-\mu_{\G,x}\tensor\mu_{\G,y}}&=0.
	\end{align}
In a prior paper~\cite{BP} we showed that the condition (\ref{eqRSG}) implies that the standard messages
$(\mu_{\G,x\to a},\mu_{\G,a\to x})_{x\in V_n,a\in\partial x}$ are approximate Belief Propagation fixed points.
In fact, under a mild additional assumption the free energy $\frac1n\Erw[\ln Z(\G)]$ can be computed from the standard messages.
However, from a practical viewpoint it is, of course, highly challenging to verify the replica symmetry condition (\ref{eqRSG}).
Hence, by comparison to~\cite{BP} the achievement of the present work is that the main results hold universally, in the replica symmetric phase and beyond.

According to the cavity method there are several significantly different scenarios of how the Bethe state decomposition might look in models where the replica symmetry condition (\ref{eqRSG}) is violated~\cite{Marinari}.
First, it could be that there is a decomposition into a {\em bounded} number of Bethe states.
The ferromagnetic Ising model on the random graph is conjectured to be an example~\cite{demboPotts}.
But in other models a bounded number of Bethe states is not expected to suffice.
The Potts antiferromagnet on the \Erdos-\Renyi\ graph for $d>\dc(q,\beta)$ is conjectured to be an example of this case, which is called {\em static replica symmetry breaking} in physics jargon~\cite{pnas,MM}.

In some models, such as the Potts antiferromagnet on the \Erdos-\Renyi\ graph, the Gibbs measure may possess a decomposition into a large number of microscopic ``clusters'' even for $d<\dc(q,\beta)$, i.e., within the replica symmetric phase~\cite{pnas}.
Although each of these clusters only carries an exponentially small probability mass,
the clusters are conjectured to be $(\eps,\ell)$-Bethe state for certain $\eps=\eps(n)\to0$, $\ell=\ell(n)\to\infty$.
In fact, under certain assumptions on $q,d,\beta$ this has been established rigorously~\cite{Nor}.
The cluster decomposition has a significant impact on the mixing time of Markov chains such as Glauber dynamics, and therefore this scenario is referred to as {\em dynamic replica symmetry breaking}.
At first glance the existence of such an exponentially large microscopic decomposition might seem to contradict the assertion of \Thm~\ref{Thm_BetheStates} that the Gibbs measure merely decomposes into an arbitrarily slowly growing number of Bethe states.
Yet for $d<\dc(q,\beta)$ the entire configuration space $\Omega^{V_n}$ itself is a Bethe state~\cite{pnas}.
Hence, despite its impact on dynamics, the clustering  phenomenon is invisible from a global, thermodynamic vista.
The micro-Bethe states simply blur into one single giant Bethe state.

The diluted mean-field models studied here are considerably different than fully connected models such as the Sherrington-Kirkpatrick model, which are, by and large, far better understood. In particular, the TAP equations, the (simplified) fixed point equations that correspond to the Belief Propagation equations in the fully connected case, have been established in several cases, for instance~\cite{Auffinger}.

\subsection{Proofs}

The main idea behind the proofs of Theorem~\ref{Thm_BetheStates} and Corollary~\ref{Cor_BP} is that the random factor graph $\G$ can be constructed iteratively, adding one new variable node at a time with an appropriate random number of constraint nodes attached to it and to other randomly chosen variable nodes from the existing factor graph.  If we knew that at each step of this process the factor graph was $\eps$-symmetric for sufficiently small $\eps$, then we would know that typically the joint distribution of the variable nodes at distance two from the new variable node factorize in the old factor graph. This factorization property is exactly what occurs on an acyclic factor graph: if we remove any variable node, the variable nodes at distance two now belong to disconnected components and so their spins are independent. The same property holds when we remove any finite size cavity from an acyclic factor graph.  To obtain the approximate factorization property in the smaller factor graph we use the pinning procedure of Section~\ref{Sec_subcube}.   

It will be convenient to view the conditional measure $\mu_G^{I, \sigma}$ as the Gibbs measure of a factor graph $G^{I, \sigma}$ obtained from $G$ by adding a set of constraint nodes each attached to single variable nodes: for each $i \in I$, we add the constraint node $b_i$, joined to $i \in V(G)$, with $\psi_{b_i} (\omega) = \vecone \{ \omega = \sigma_i \}$.  These are \textit{hard constraints} that prescribe the values of the variables in $I$ to $\sigma|_I$, and thus $\mu_{G^{I,\sigma}} = \mu_G^{I,\sigma}$.   The \textit{pinned random factor graph} $\G^{\vI, \SIGMA}$ is obtained by drawing $\G$ from $\G  (n ,\vm, P)$, choosing $\SIGMA$ according to $\mu_{\G}$ and choosing $\vI \subset [n]$, $|\vI| = \THETA$, as in Section~\ref{Sec_subcube}.     

We begin  by showing that approximate factorization among boundary variable nodes suffices to obtain the local properties of an approximate Bethe state and Belief Propagation fixed point.  We will apply the following lemma to the pinned random factor graph $\G^{\vI,\SIGMA}$.

\begin{lemma}
\label{lemBethe1}
Fix a finite set of constraint functions $\Psi$. For every $\eps>0, R>0$ there is $\eps'>0$ so that the following is true. Let $G$ be any factor graph. Let $U \subset V(G)$ be a cavity of size at most $R$, $\partial U$ its boundary with $|\partial U| \le R$, and   $Y$  the set of variables nodes at distance $2$ from $U$, and suppose that all constraint nodes $a \in F(G)$ with $\partial a \cap U \ne \emptyset$ have $\psi_a \in \Psi$. Suppose that 
\begin{align}
\label{eqDecorAssumpt}
\left \| \mu_{G  \setminus  \partial U,Y} - \bigotimes_{y \in Y} \mu_{G \setminus  \partial U, y}   \right \|_{TV} < \eps' \, .
\end{align}
 Then with $\bar \mu_{G,U} [ \nix ]$ defined as in~\eqref{eqMubar}, 
 \begin{align*}
\sum_{\sigma \in \Omega^U} \left | \bar \mu_{G,U}(\sigma) - \mu_{G,U}(\sigma)  \right | < \eps \, .
\end{align*}
Moreover, suppose~\eqref{eqDecorAssumpt} holds for a cavity $U$ of size $1$, that is a variable node $x$ with $|\partial x| \le R$. Let $\mu$ denote the canonical messages $\mu_{G, y \to a}, \mu_{G, a\to y}$ and $\hat \mu= BP(\mu)$. Then
\begin{align*}
\sum_{a\in\partial x}\sum_{\sigma\in\Omega}
		\abs{\mu_{G,x\to a}(\sigma)-\hat \mu_{G,x\to a}(\sigma)}+\abs{\mu_{G,a\to x}(\sigma)-\hat \mu_{G,a\to x}(\sigma)} <\eps \, .
\end{align*}
\end{lemma}

\begin{proof}
For this proof, we let $o(1)$ denote a term that tends to $0$ with $\eps'$, where the implied constant depends on $R$. 
We can write 
\begin{align*}
 \mu_{G,U}(\sigma) &\propto 
 	{ \sum_{\tau \in \Omega^{U \cup Y}} \vecone\{\tau|_U =\sigma\}  \cdot \mu_{G  \setminus \partial U,Y}(\tau|_Y) \prod_{a \in \partial U}  \psi_a(\tau|_{\partial a}) \prod_{a \in G[U]} \psi_a(\sigma|_{\partial a}) }
	\\
 &\propto
 \sum_{\tau \in \Omega^{U \cup Y}} \vecone\{\tau|_U =\sigma\}  \cdot  \prod_{y \in Y} \mu_{G  \setminus \partial U,y}(\tau_y) \prod_{a \in \partial U}  \psi_a(\tau|_{\partial a}) \prod_{a \in G[U]} \psi_a(\sigma|_{\partial a}) 
 + o(1) \\
 &\propto \mu_{G[U]}(\sigma) \cdot \prod_{a \in \partial U}  \sum_{\tau \in \Omega^{\partial a}} \mathbf 1 \{ \tau_{x_{a\to U}} = \sigma_{x_{a\to U}} \} \cdot \psi_a(\tau) \prod_{y \in \partial a \setminus U}  \mu_{G  \setminus \partial U,y}[ \tau_y]  +o(1) \,,
\end{align*}
where the second equation is obtained by the assumptions of the lemma. To obtain the first conclusion of the lemma it suffices to show that for each $a \in \partial U, \omega \in \Omega$, with $x_a=x_{a\to U}$ we have
\begin{align*}
\mu_{a \to x_a}(\omega) &\propto  \sum_{\tau \in \Omega^{\partial a}} \mathbf 1 \{ \tau_{x_a} = \omega \} \cdot \psi_a(\tau) \prod_{y \in \partial a \setminus U}  \mu_{G  \setminus \partial U,y}[ \tau_y]  +o(1) \, ,
\end{align*}
where  $\mu_{a \to x_a}$ is the standard message defined to be the marginal of $x_a$ in the factor graph obtained  by removing all constraint nodes $b\in\partial x_a$ except $a$.  We can write
\begin{align*}
\mu_{a \to x_a}(\omega) &\propto 
\sum_{\tau \in \Omega^{U \cup Y}} \vecone\{\tau_{x_a} =\omega\}  \cdot \mu_{G  \setminus \partial U,Y}(\tau|_Y)  \psi_a (\tau|_{\partial a}) \prod_{b \in \partial U \setminus \partial x_a}  \psi_b(\tau|_{\partial b}) \prod_{b \in G[U] \setminus \partial x_a} \psi_b(\tau|_{\partial b})  
\\
&\propto
\sum_{\tau \in \Omega^{U \cup Y}} \vecone\{\tau_{x_a} =\omega\}  \psi_a (\tau|_{\partial a}) \prod_{y \in \partial a \setminus U} \mu_{G  \setminus \partial U,y}[ \tau_y] \cdot \mu_{G  \setminus \partial U,Y \setminus \partial a}(\tau|_{Y\setminus \partial a})  \prod_{b \in \partial U \setminus \partial x_a}  \psi_b(\tau|_{\partial b})\hspace{-2mm} \prod_{b \in G[U] \setminus \partial x_a} \hspace{-2mm}\psi_b(\tau|_{\partial b})  
+o(1) \\
&\propto  \sum_{\tau \in \Omega^{\partial a}} \mathbf 1 \{ \tau_{x_a} = \omega \} \cdot \psi_a(\tau) \prod_{y \in \partial a \setminus U}  \mu_{G  \setminus \partial U,y}[ \tau_y]  +o(1) \,,
\end{align*}
since the expression 
\begin{align*}
 \mu_{G  \setminus \partial U,Y \setminus \partial a}(\tau|_{Y\setminus \partial a})  \prod_{b \in \partial U \setminus \partial x_a}  \psi_b(\tau|_{\partial b}) \prod_{b \in G[U] \setminus \partial x_a} \psi_b(\tau|_{\partial b}) 
\end{align*}
does not depend on $\tau|_{\partial a}$. 

The second part of the lemma is similar. We first show that for $a \in \partial x, y \in Y$, with $y\sim a$, $\mu_{G,y \to a}(\omega) = \mu_{G \setminus  \partial x,y}(\omega) + o(1)$:
\begin{align*}
\mu_{G, y \to a}(\omega) &= \mu_{G \setminus a,y}(\omega) 
\propto  
	\sum_{\tau \in \Omega^{Y \cup x}} \vecone\{\tau_y =\omega\} \mu_{G \setminus  \partial x,Y}(\tau|_Y) \prod_{b \in \partial x} \psi_b(\tau|_{\partial b}) 
	\\
&\propto 
\mu_{G \setminus \partial x,y}(\omega) \sum_{\tau \in \Omega^{(Y \cup x) \setminus y}}  \mu_{G \setminus \partial x,Y\setminus y}(\tau|_{Y \setminus y}) \prod_{b \in \partial x} \psi_b(\tau|_{\partial b})
+ o(1) = \mu_{G \setminus  \partial x,y}(\omega)  + o(1) \,.
\end{align*}
Next we see that 
\begin{align*}
\mu_{G, a\to x}(\omega) &=  \mu_{G \setminus(\partial x \setminus a),x} (\omega)
\propto
	\sum_{\tau \in \Omega^{Y \cup x}} \vecone\{\tau_x =\omega\} \mu_{G \setminus \partial x,Y}(\tau|_Y)  \psi_a(\tau|_{\partial a}) 
	\\
&\propto
	\sum_{\tau \in \Omega^{\partial a}}  \vecone\{\tau_x =\omega\}\psi_a(\tau|_{\partial a})  \prod_{y \in \partial a \setminus x} \mu_{G,y\to a}(\tau_y) 
	 + o(1)  
= \hat \mu_{G,a\to x}(\omega) +o(1) \,. 
\end{align*}
Finally,
\begin{align*}
\mu_{G, b \to x}(\omega)&= \mu_{G \setminus(\partial x \setminus b),x} (\omega)
\propto
	\sum_{\tau \in \Omega^{Y \cup x}} \vecone\{\tau_x =\omega\} \mu_{G \setminus \partial x,Y}(\tau)   \psi_b(\tau|_{\partial b }) 
	\\
&= 
	\sum_{\tau \in \Omega^{\partial b}} \vecone\{\tau_x =\omega\}    \psi_b(\tau) \prod_{y \in \partial b \setminus x}\mu_{G \setminus \partial x,y}(\tau_y) 	 +o(1)
\end{align*}
and so
\begin{align*}
\mu_{G, x \to a}(\omega) &\propto  
	\sum_{\tau \in \Omega^{Y \cup x}} \vecone\{\tau_x =\omega\} \mu_{G \setminus  \partial x,Y }(\tau|_Y)  \prod_{b \in \partial x \setminus a} \psi_b(\tau|_{\partial b })  
	\\
&= 
 \prod_{ b \in \partial x \setminus a} \sum_{\tau \in \partial b} \vecone\{\tau_x =\omega\} \psi_b(\tau) \prod_{y \in \partial b \setminus x}\mu_{G \setminus \partial x,y}(\tau_y) 
 +o(1)    
\propto  \prod_{ b \in \partial x \setminus a} \mu_{G, b \to x}(\omega)  +o(1) \, ,
\end{align*}
and so $\mu_{G, x \to a}(\omega) = \hat \mu_{G, x\to a}(\omega) + o(1)$. 
\end{proof}

We next show that with high probability over $\G^{\vI,\SIGMA}$ and the choice of a random cavity $\vU$ from $\G^{\vI,\SIGMA}$, the assumptions of Lemma~\ref{lemBethe1} hold.  We first set up a coupling of two random factor graphs.

\begin{lemma}
\label{lemCoupling}
Fix $1 \le r \le l$. Then the following distributions  on pairs of factor graphs and cavities, $(\G, \vU)$ and $(\G', \vU')$ have total variation distance $o(1)$:
\begin{enumerate}
\item Choose $\G$ according to the random factor graph model $\G(n, \vec m, P)$ and choose a cavity $\vU$ uniformly at random from $\mathcal C(\G,  l, r  )$.  Let $\mathcal D_{l,r,n}$ be the distribution of $\G[\vU]$.
\item Form a factor graph on $n$ variables nodes as follows. Choose  
	$$\vec m' \sim \Po\left (\frac{dn}{k} \left (1-  (l/n)^k\right)  \right)$$ and add $\vec m'$ constraint nodes $a_1, \dots a_{\vec m'}$ with $\partial a_i$ chosen uniformly from $x_1, \dots x_{n}$ conditioned on the event that not all nodes are chosen from $x_{n-l+1}, \ldots x_n$.  On the variable nodes $x_{n-l+1}, \dots x_n$ place a factor graph drawn according to  $\mathcal D_{l,r,n}$.   Permute the labels of the variable nodes at random. Call the resulting factor graph $\G'$ and let $\vU'$ denote the variable nodes originally labeled $x_{n-l+1}, \dots x_n$. 
\end{enumerate}
\end{lemma}
\begin{proof}
This is essentially a standard fact from random graph theory.  Note that in both distributions the constraint functions assigned to each constraint node are chosen independently from $P$ for all constraint nodes so it suffices to prove the statement for the distributions of the underlying bipartite graph and selected cavity.  Moreover, as the distribution of the subgraph $\G'[\vU]$ is by definition the same as $\G[\vU]$, we can fix the structure of the subgraph, call this $H$.  Finally the definition of a cavity requires that any $a \in \partial U$ intersects $U$ in exactly one variable node: in the second distribution one of the $\vec m' $ constraint  nodes adjoins two or more variable nodes from $x_{n-l+1}, \ldots x_n$ with probability $O(1/n)$ so we may condition on the event that this does not occur.

It is somewhat easier to work with a slightly different interpretation of the random bipartite graph induced by the random factor graph model: for each of the $(n)_k$ ordered sets of $k$ variable nodes, include a constraint node attached to this set with probability $p = 1- \exp \left ( \frac{ dn}{k}  \frac{1}{(n)_k} \right) \sim \frac{d}{k n^{k-1}}$, independently for each ordered set of $k$ nodes, where $(n)_k = n (n-1) \cdots (n-k+1)$.  Then for each constraint node that appears in this graph, choose a multiplicity according to the distribution $X \sim \Po\left ( \frac{ dn}{k}  \frac{1}{(n)_k} \right) $ conditioned on $X \ge 1$. In this way each possible constraint node appears a $\Po\left ( \frac{ dn}{k}  \frac{1}{(n)_k} \right)$ number of times and by Poisson thinning this gives the same distribution as including $\Po(dn/k)$ constraint nodes with boundaries chosen uniformly at random.   To prove the lemma we can ignore multiplicities as with probability $1-O(1/n)$ no constraint node adjoining a variable node in $\vU$ or $\vU'$ will have multiplicity more than $1$, and the multiplicities in the rest of the graph can be coupled exactly.  So we consider simply the random bipartite graphs in which each possible constraint node appears independently with probability $p$.  

 Let $Q, Q'$ denote the respective probability distributions on pairs.  We can calculate the respective probabilities
\begin{align*}
Q[ G, U] &= \vecone\{G[U] = H  \} p^{|F(G)|}(1-p)^{(n)_k - |F(G)|} \frac{1}{\# H(G)} \\
Q'[ G, U]&=  \vecone\{G[U] = H  \} \frac{1}{(n)_l} p^{|F(G)|-|F(H)|} (1-p) ^{(n)_k- |F(G)|+ (l)_k -|F(H)|  }  \,,
\end{align*}
 where $\# H(G)$ is the number of induced copies of $H$ in $G$. 
 
 We can write the total variation distance as
 \begin{align*}
\| Q- Q'\|_{TV} &= \Erw_{Q} \left[ \left |1- \frac{Q'(\G,\vU)}{Q(\G,\vU)} \right |  \right] \\
&= \Erw_{Q} \left[ \left |1- \frac{(n)_l p^{|F(H)|} (1-p)^{(l)_k - |F(H)|}}{\# H(\G)} \right |  \right] \,. 
\end{align*}
The numerator in the fraction is exactly $\Erw_{Q}[ \# H(\G)]$.  Since $H$ is of bounded size and is acyclic, $\Erw_{Q}[ \# H(\G)]=\Theta(n)$, and the random variable $\# H(\G)$ has standard deviation of order $n^{1/2}$. Thus
\begin{align*}
\| Q- Q'\|_{TV}= \Erw_{Q} \left[ \left |1- \frac{(n)_K p^{|F(H)|} (1-p)^{\binom{K}{k} - |F(H)|}}{\# H(\G)} \right |  \right]  &= o(1) \, ,
\end{align*}
as desired.
 \end{proof}

\begin{lemma}
\label{lemDecor}
For every $\eps>0, K>0$ there exists a bounded random variable $\THETA$ so that the following is true for all $r \le l \le K$. Select $\G^{\vI, \hat\SIGMA}$  from the pinned random factor graph model with $|\vI| = \THETA$. Choose $\vU$ uniformly from the set of cavities $\mathcal C ( \G,l,r)$. Let $\vec Y$ be  the set of variables nodes at distance $2$ from $\vU$. 
  With probability at least $1-\eps$ over the choice of $\vI, \hat\SIGMA, \vU$,
\begin{align*}
\left \| \mu_{\G^{\vI, \hat\SIGMA} \setminus \partial \vU,\vec Y} - \bigotimes_{y \in \vec Y} \mu_{\G^{\vI, \hat\SIGMA} \setminus  \partial \vU, y}   \right \|_{TV} < \eps \, .
\end{align*}
\end{lemma}
\begin{proof}
We form $\G^{\vI, \hat\SIGMA}$ in stages using Lemma~\ref{lemCoupling} and the pinning procedure. 

We form two factor graphs on $n$ variable nodes, $\check \G$ and $\G'$. The factor graph $\check \G$ will have $l$ isolated variable nodes on which the cavity $\vU$ will be placed in the factor graph $\G'$, along with constraint nodes in $\partial \vU$.  We will show that when $\hat \SIGMA$ is chosen according to $\mu_{\G'}$, the measure $\mu_{\check \G}^{\vI, \hat \SIGMA}$ is $\eps'$-symmetric with high probability, and from this we derive the conclusion of the lemma.   

 Let $\vec m' \sim \Po\left (\frac{dn}{k} \left (1-  (l/n)^k\right)  \right)$, and choose $\vec m'$ constraint nodes with boundaries chosen uniformly from the variable nodes conditioned on the event they are not all among the last $l$ variables.  Let $\check \G$ be the factor graph on $x_1, \dots x_n$ formed by including only those constraint nodes that do not adjoin any variable node in $x_{n-l+1}, \dots x_n$. Let $\G'$ be the factor graph consisting of all $\vec m'$ constraint nodes and the addition of   constraint nodes  chosen according to $\mathcal D_{l,r,n}$ incident to $U= \{x_{n-l+1}, \dots x_n \}$ c.   

 Now choose $\vI \subset [n]$, $|\vI| = \THETA$ and choose $\hat \SIGMA $ from $\mu_{\G'}$.  By Lemma~\ref{lemCoupling}, up to total variation distance $o(1)$, $ (\G')^{\vI, \hat \SIGMA}$ has the distribution of $\G^{\vI, \hat\SIGMA}$ and $U$ is distributed as a uniformly random cavity from $\mathcal C(\G(n,\vec m),l,r)$.   

Condition on the event that the total number of constraint nodes incident to variables  in $ U$   is at most $L Kd$ and each constraint node in $\partial U$ joins exactly one $x \in U$.  For $L=L(\eps)$ chosen large enough that this happens with probability at least $1-  \eps/3$.  Under this event we have $|Y(U)| \le (k-1) LKd$.  Note that for $a \in \partial U$, the the variable nodes in $\partial a \setminus U$ are chosen uniformly at random from $x_1, \dots x_{n-l}$.   

For $I \subset [n]$, let $B_{I}$ be the set of $\tau \in \Omega^{I}$ so that 
\begin{align*}
	\frac1{n^R}\sum_{1\leq i_1<\cdots<i_R\leq n}\TV{\mu^{I, \tau}_{\check \G,i_1,\ldots,i_R}-\mu^{I, \tau}_{\check \G,i_1}\tensor\cdots\tensor\mu^{I, \tau}_{\check \G,i_R}}&<\eps \, ,
\end{align*}
with $R = (k-1) LKd$.  Let $\eps'$ be small enough as a function of $\eps$. Using Theorem~\ref{Thm_pin} and \Prop~\ref{Cor_symmetry}, the probability that $\vI$ is such that $\mu_{\check \G,\vI}(B_{\vI}) >1- \eps'$ is at least $1-\eps/3$ if we choose the distribution of $\THETA$ appropriately as a function of $\eps'$.  Fix such a set $J$. Now $\mu_{\check \G}$ and $\mu_{\G'}$ have different distributions, but they are contiguous: since all constraint functions are positive, and $\check \G$ and $\G'$ differ by at most $LKd$ constraint nodes, 
\begin{align*}
\mu_{\G'}(\sigma) &\le c^{LKd} \mu_{\check \G}(\sigma)
\end{align*}
for some constant $c = c(\Omega, \Psi)$ and all $\sigma \in \Omega^{V_n}$.  This implies that $\mu_{\G',J}(B_J) >1- c^{LKd} \eps' > 1-\eps/3$ if $\eps'$ is chosen small enough as a function of $\eps, d, K$.  Now since the variables nodes in $Y(U)$ are chosen uniformly at random from $x_1, \dots  x_{n-l}$, the conclusion follows.
\end{proof}

Using the ingredients above we complete the proof of Theorem~\ref{Thm_BetheStates} and Corollary~\ref{Cor_BP}.

\begin{proof}[Proof of Theorem~\ref{Thm_BetheStates} and Corollary~\ref{Cor_BP}]

To prove Theorem~\ref{Thm_BetheStates} it suffices to find a function $L= L(\eps, K)$ for every $\eps>0, K>0$ so that for large enough $n$, we can find the desired decomposition of $\G(n, \vec m, P)$ into at most $L$ parts. 

We find such a decomposition of $\Omega^{V_n}$ via the subcube decomposition given by the pinning operation.  We choose $\vI \subset [n]$ at random and consider the decomposition into states $S^{\vI, \tau}, \tau \in \Omega^{\vI}$.   The conclusions of Theorem~\ref{Thm_BetheStates} can be interpreted probabilistically: with probability at least $1-\eps$ over the choice of $\vI$, $\hat \SIGMA \sim \mu_{\G}$, and $\vU \in \mathcal C(\G, l,r)$,
 \begin{align}
 \label{eqProbstate}
\sum_{\sigma \in \Omega^{\vU}} \left | \bar \mu_{ \G^{\vI,\SIGMA},\vU}(\sigma) - \mu_{\G^{\vI,\SIGMA},\vU}(\sigma)  \right | < \eps \, .
\end{align}

To show~\eqref{eqProbstate}, fix $\eps>0, K>0$ and choose $\eps'>0$ small enough.  Let $Y(U)$ denote the set of variable nodes at distance $2$ from a cavity $U$. For a given set $I \subseteq [n]$, let 
\begin{align*}
A_I = \left \{\tau \in \Omega^I: \frac{1}{| \mathcal C ( \G,l,r)|} \sum_{U \in \mathcal C ( \G^{I,\tau},l,r)} \left \| \mu_{ \G^{I,\tau}  \setminus \partial U, Y(U)} - \bigotimes_{y \in  Y(U)} \mu_{ \G^{I,\tau} \setminus  \partial U, y}   \right \|_{TV} < \eps'     \, \text{ for all } 1\le r \le l\le K  \right \} \, .
\end{align*}
 By Lemma~\ref{lemDecor} (and a union bound over all $r,l$) there is a set $I \subset [n]$, of  bounded size in terms of $\eps$, so that  $\mu_{ \G, I}(A_I) > 1-\eps$. We claim that the subcube decomposition of $\Omega^{V_n}$ along $I$ gives us our desired decomposition, with the `good' parts of the decomposition given by $S^{I,\tau}$ with $\tau \in A_I$. By our choice of $I$ we have $\sum_{\tau \in A_I} \mu(S^{I,\tau}) > 1-\eps$, as desired.  
 
 We now show that for $\tau \in A_I$, $S^{I,\tau}$ is an $(\eps, l)$-Bethe state for $i=1,\dots K$. The definition of $A_I$ implies that for each $1\le r\le l \le K$, and for $1-\eps'$ fraction of $U$ chosen from $\mathcal C( \G, l,r)$, we have 
\begin{align*}
\left \| \mu_{G^{I,\tau}  \setminus  \partial \vU,Y(U)} - \bigotimes_{y \in Y(U)} \mu_{G^{I,\tau} \setminus  \partial U, y}   \right \|_{TV} < \eps' \, .
\end{align*}
Then Lemma~\ref{lemBethe1} tells us that 
 \begin{align*}
\sum_{\sigma \in \Omega^U} \left | \bar \mu_{G^{I,\tau},U}(\sigma) - \mu_{G^{I,\tau},U}(\sigma)  \right | < \eps \, ,
\end{align*}
and so $S^{I,\tau}$ is an $(\eps,l)$-Bethe state for $i=1,\dots K$, proving Theorem~\ref{Thm_BetheStates}. Moreover, let $\mu_{x \to a} [ \nix | S^{I,\tau}], \mu_{a \to x}[ \nix |S^{I,\tau}]$ denote the canonical message set of $\mu_{ \G}[\nix |S^{I,\tau}]$, and let $\hat \mu_{a \to x}, \hat \mu_{x \to a}$ be their image under the BP operator.  Taking $r=l=1$, and applying the second part of Lemma~\ref{lemBethe1}, we obtain that 
\begin{align*}
\frac1{n}\sum_{x\in V( \G)}\sum_{a\in\partial x}\sum_{\sigma\in\Omega}
		\abs{\mu_{x\to a}(\sigma| S^{I,\tau})-\hat \mu_{x\to a}(\sigma)}+\abs{\mu_{a\to x}(\sigma|S^{I,\tau})-\hat \mu_{a\to x}(\sigma)} < \eps
\end{align*}
and so the canonical messages of $\mu(\nix |S^{I,\tau})$ are an $\eps$-Belief Propagation fixed point, proving Corollary~\ref{Cor_BP}.
\end{proof}

\begin{proof}[Proof of \Cor~\ref{Cor_BetheCutMetric}]
The assertion is immediate from \Thm~\ref{Thm_pin} and \Prop~\ref{Cor_symmetry}.
\end{proof}

\subsection*{Acknowledgments.}
We thank Dmitry Panchenko for many detailed comments and corrections on an early version of this paper. We also thank Ryan O'Donnell for bringing~\cite{Allen,Raghavendra} to our attention and Florent Krzakala and Guilhem Semerjian for helpful comments.  WP was supported in part by the Engineering and Physical Sciences Research Council (EPSRC) grant EP/P009913/1.

\end{document}